\documentclass[11pt]{amsart}

\title{Unipotent Representations and Microlocalization}

\author{Lucas Mason-Brown}
\usepackage{amsmath,amssymb,amsthm,amsfonts,graphics ,tikz-cd, mdframed,enumitem, appendix, scalerel, stackengine,todonotes,comment}

\calclayout

\numberwithin{equation}{section}

\newtheorem{theorem}{Theorem}[section]
\newtheorem{theorem*}{Theorem}
\newtheorem{rmk}[theorem]{Remark}

\newtheorem{conj}[theorem]{Conjecture}

\newtheorem{prop}[theorem]{Proposition}%[chapter]
\newtheorem{definition}[theorem]{Definition}%[chapter]
\newtheorem{cor}[theorem]{Corollary}%[chapter]
\newtheorem{lemma}[theorem]{Lemma}%[chapter]
\newtheorem{lemma*}[theorem*]{Lemma}%[chapter]

\newtheorem{example}[theorem]{Example}

\newcommand{\ZZ}{\mathbb{Z}}

\newcommand{\RR}{\mathbb{R}}
\newcommand{\CC}{\mathbb{C}} 
\newcommand{\fg}{\mathfrak{g}} 
\newcommand{\fk}{\mathfrak{k}} 
\newcommand{\fh}{\mathfrak{h}} 

\newcommand{\gr}{\operatorname{gr}}
\newcommand{\Coh}{\operatorname{Coh}}
\newcommand{\QCoh}{\operatorname{QCoh}}
\newcommand{\codim}{\operatorname{codim}}
\newcommand{\AV}{\operatorname{AV}}
\newcommand{\HC}{\operatorname{HC}}
\newcommand{\cN}{\mathcal{N}}
\newcommand{\OO}{\mathbb{O}}

\newcommand{\Ind}{\mathrm{Ind}} 

\stackMath
\newcommand\reallywidehat[1]{%
\savestack{\tmpbox}{\stretchto{%
  \scaleto{%
    \scalerel*[\widthof{\ensuremath{#1}}]{\kern-.6pt\bigwedge\kern-.6pt}%
    {\rule[-\textheight/2]{1ex}{\textheight}}%WIDTH-LIMITED BIG WEDGE
  }{\textheight}% 
}{0.5ex}}%
\stackon[1pt]{#1}{\tmpbox}%
}

\begin{document}
\begin{abstract}
We develop a theory of microlocalization for Harish-Chandra modules, adapting a construction of Losev (\cite{Losev2011}). We explore the applications of this theory to unipotent representations of real reductive groups. For complex groups, we deduce a formula for the $K$-multiplicities of unipotent representations attached to a nilpotent orbit $\OO$, proving an old conjecture of Vogan (\cite{Vogan1991}) in a large family of cases.
\end{abstract}

\maketitle

\tableofcontents

\section{Introduction}\label{sec:intro}

Let $G_{\RR}$ be the real points of a connected reductive algebraic group. In \cite{AdamsBarbaschVogan}, Adams, Barbasch, and Vogan, following ideas of Arthur (\cite{Arthur1983},\cite{Arthur1989}), introduced a finite set of irreducible representations of $G_{\RR}$, called \emph{special unipotent representations}. We will recall their definition in Section \ref{subsec:unipotent}. These representations are conjectured to possess an array of distinguishing properties (see \cite[Chp 1]{AdamsBarbaschVogan}). For example:
\begin{itemize}
    \item They are conjectured to be unitary.
    \item They are conjectured to appear in spaces of automorphic forms.
    \item They are conjectured to \emph{generate} (through various types of induction) all irreducible unitary representations of $G$ of integral infinitesimal character
\end{itemize}
Now let $K_{\RR} \subset G_{\RR}$ be a maximal compact subgroup. Any (nice) irreducible representation $X$ of $G_{\RR}$ decomposes as a $K_{\RR}$-representation into irreducible components, each with finite multiplicity. The general philosophy of unipotent representations suggests that if $X$ is unipotent, then these multiplicities should be `small'. In \cite[Conj 12.1]{Vogan1991}, Vogan offers a conjectural description of the restriction to $K_{\RR}$ of a unipotent representation (under some additional conditions). A little more precisely, he conjectures

\begin{conj}
Let $X$ be a unipotent representation of $G_{\RR}$. Then there is a subgroup $H_{\RR} \subset K_{\RR}$ and a finite-dimensional $H_{\RR}$-representation $e$ such that
$$X \simeq_{K_{\RR}} \Ind^{K_{\RR}}_{H_{\RR}} \chi.$$
\end{conj}
Of course, the actual conjecture in \cite{Vogan1991} is much more precise, see Conjecture \ref{conj:Vogan} below. In this paper, we will prove Vogan's conjecture in a large family of cases.

The main ingredient in our proof is a functor $\Phi_{\OO}$ which `microlocalizes' Harish-Chandra modules over a nilpotent $K$-orbit $\OO$. The construction of $\Phi_{\OO}$ follows \cite[Sec 4]{Losev2011}, where a similar functor is constructed for Harish-Chandra bimodules.

\subsection{Acknowledgements} The author would like to thank David Vogan and Ivan Losev for many illuminating conversations. Special thanks to Ivan Losev for patiently explaining the details of his extension and restriction functors for Harish-Chandra bimodules.

\section{Preliminaries}\label{sec:preliminaries}

Fix $K_{\RR} \subset G_{\RR}$ as in Section \ref{sec:intro}. Write $K \subset G$ for the complexifications and $\mathfrak{k} \subset \fg$ for the (complex) Lie algebras. Let $\theta: \fg \to \fg$ denote the Cartan involution corresponding to $\mathfrak{k}$ and let $\mathfrak{p} \subset \mathfrak{g}$ be the $-1$-eigenspace of $\theta$. There is a Cartan decomposition $\mathfrak{g} = \mathfrak{k} + \mathfrak{p}$ which is orthogonal with respect to (any choice of) an invariant symmetric bilinear form on $\mathfrak{g}$.

A $(\fg,K)$-module is a left module $X$ for the universal enveloping algebra $U(\fg)$ of $\fg$ together with an algebraic $K$-action such that
\begin{enumerate}
\item The action map $U(\mathfrak{g}) \otimes X \to X$ is $K$-equivariant,
\item The $\mathfrak{k}$-action on $X$, coming from the inclusion $\mathfrak{k} \subset \mathfrak{g} \subset U(\mathfrak{g})$, coincides with the differentiated action of $K$.
\end{enumerate}
A morphism of $(\fg,K)$-modules is a $U(\fg)$-module homomorphism which intertwines the $K$-actions. Let $M(\fg,K)$ denote the (abelian) category of $(\fg,K)$-modules. A \emph{Harish-Chandra module} is a $(\fg,K)$-module which is finitely generated for $U(\fg)$. Note that an irreducible $(\fg,K)$-module is automatically Harish-Chandra. Write $\HC(\fg,K) \subset M(\fg,K)$ for the full subcategory of Harish-Chandra modules.

\subsection{Unipotent representations}\label{subsec:unipotent}

Let $G^{\vee}$ denote the Langlands dual of $G$, and write $\cN \subset \fg^*$, $\cN^{\vee} \subset (\fg^{\vee})^*$ for the nilpotent cones. The nilpotent orbits for $G$ and $G^{\vee}$ are related via \emph{Barbasch-Vogan duality} (see \cite{BarbaschVogan1985}). This is a map
$$d: \{\text{nilpotent orbits } \OO^{\vee} \subset \cN^{\vee}\} \to \{\text{nilpotent orbits } \OO \subset \cN\}.$$
A nilpotent orbit $\OO \subset \cN$ is \emph{special} if it lies in the image of $d$.

Every nilpotent $G^{\vee}$-orbit $\OO^{\vee} \subset \cN^{\vee}$ gives rise to an infinitesimal character $\lambda_{\OO^{\vee}}$ for $U(\mathfrak{g})$ as follows. If we fix a Cartan subalgebra $\fh \subset \fg$, there is a Cartan subalgebra $\fh^{\vee} \subset \fg^{\vee}$, which is canonically identified with $\fh^*$. Using a $G^{\vee}$-invariant identification $\fg^{\vee} \simeq (\fg^{\vee})^*$, we can regard $\OO^{\vee}$ as a nilpotent $G^{\vee}$-orbit in $\fg^{\vee}$. Choose an element $e^{\vee} \in \OO^{\vee}$ and an $\mathfrak{sl}(2)$-triple $(e^{\vee},f^{\vee},h^{\vee})$. Conjugating by $G^{\vee}$ if necessary, we can arrange so that $h^{\vee} \in \mathfrak{h}^{\vee} \simeq \fh^*$. Put
$$\lambda_{\OO^{\vee}} := \frac{1}{2}h^{\vee} \in \mathfrak{h}^{\vee} \simeq \mathfrak{h}^*.$$
This element is well-defined modulo the (linear) action of the Weyl group and thus determines an infinitesimal character for $U(\mathfrak{g})$ by means of the Harish-Chandra isomorphism. Let $I_{\OO^{\vee}}\subset U(\fg)$ be the unique maximal ideal with infinitesimal character $\lambda_{\OO^{\vee}}$. A \emph{special unipotent ideal} is any ideal in $U(\fg)$ which arises in this fashion.

Write $\AV(I) \subset \fg^*$ for the associated variety of a two-sided ideal $I \subset U(\fg)$. By \cite[Prop A2]{BarbaschVogan1985}, we have
$$\AV(I_{\OO^{\vee}}) = \overline{d(\OO^{\vee})}.$$
In particular, the associated variety of a special unipotent ideal is (the closure of) a special nilpotent orbit (this explains the word `special' in `special unipotent').

\begin{definition}\label{def:specialunipotent}
Suppose $\OO^{\vee} \subset \cN^{\vee}$ is a nilpotent orbit. A \emph{special unipotent representation} attached to $\OO^{\vee}$ is an irreducible $(\fg,K)$-module $X$ such that
$$\mathrm{Ann}_{U(\fg)}(X) = I_{\OO^{\vee}}.$$
\end{definition}

Special unipotent ideals belong to a larger class of maximal ideals called simply \emph{unipotent ideals}. This more general class of ideals is defined in the forthcoming paper \cite{LMBM}. If $I$ is a unipotent ideal, one can define the notion of a \emph{unipotent representation} analogously to Definition \ref{def:specialunipotent}. Our main results (Theorem \ref{thm:Vogangeom} and Corollary \ref{cor:Vogancomplex}) apply, without modification, to this more general class of representations (and the proofs are identical). Since we will not recall here the definition of this more general class of representations, the reader may choose to interpret `unipotent' to mean `special unipotent' wherever it is used below.

\subsection{Associated varieties and associated $K$-cycles}\label{subsec:associatedstuff}

Following \cite{Vogan1991}, we will associate to every Harish-Chandra module $X$ some geometric data in $\cN$. We will need the concept of a \emph{good filtration} of $X$. A filtration of $X$
$$...\subseteq X_{-1} \subseteq X_0 \subseteq X_1 \subseteq ... , \qquad \bigcap_m X_m = 0, \qquad \bigcup_m X_m = X$$
by complex subspaces is \emph{compatible} if
\begin{enumerate}
\item $U_m(\mathfrak{g})X_n \subseteq X_{m+n}$ for every $m,n \in \ZZ$.
\item $KX_m \subseteq X_m$ for every $m \in \ZZ$.
\end{enumerate}
Under these conditions, $\gr(X)$ has the structure of a graded, $K$-equivariant $S(\mathfrak{g}/\mathfrak{k})$-module. Our compatible filtration is \emph{good} if
\begin{enumerate}[resume]
\item\label{cond3} $\gr(X)$ is finitely-generated over $S(\mathfrak{g})$.
\end{enumerate}
There is an equivalence of categories (obtained by taking global sections) between the category of $K$-equivariant coherent sheaves $\Coh^K(\fg/\fk)^*$ on the affine space $(\fg/\fk)^*$ and the category of finitely-generated $K$-equivariant $S(\fg/\fk)$-modules. Thus if $X$ is equipped with a good filtration, we can (and will) regard $\gr(X)$ as an object in $\Coh^K(\fg/\fk)^*$. The following is standard (see \cite[Prop 2.2]{Vogan1991} for a proof).

\begin{prop}\label{prop:grprop}
Every $(\fg,K)$-module admits a good filtration. 
The passage from $X \in \HC(\fg,K)$ to $\gr(X) \in \Coh^K(\fg/\fk)^*$ induces a homomorphism on the Grothendieck groups
$$
KHC(\mathfrak{g},K) \to K\Coh^K(\mathfrak{g}/\mathfrak{k})^*.
$$
\end{prop}

Proposition \ref{prop:grprop} provides a recipe for attaching geometric invariants to Harish-Chandra modules. A function $\varphi: \Coh^K(\mathfrak{g}/\mathfrak{k})^* \to S$ with values in a semigroup $S$ is \emph{additive} if $\varphi(B) = \varphi(A)+\varphi(C)$ whenever there is a short exact sequence $0 \to A \to B \to C \to 0$. Under this condition, $\varphi$ is well-defined on classes in $KHC(\mathfrak{g},K)$ and therefore (by Proposition \ref{prop:grprop}), induces an (additive) function $\varphi[\gr (X)]$ on Harish-Chandra modules. 

The first example of this construction is the \emph{associated variety} $\AV(X)$ of a Harish-Chandra module $X$. Let $S$ be the set of Zariski-closed subsets of $(\mathfrak{g}/\mathfrak{k})^*$ with addition defined by $\cup$. Let $\varphi$ be the function
$$\varphi: \Coh(\mathfrak{g}/\mathfrak{k})^* \to S, \qquad \varphi(M) = \mathrm{Supp}(M) = V(\mathrm{Ann}(M)).$$
Since support is additive, $\varphi$ induces an (additive) function on $K\HC(\fg,K)$. If $X \in \HC(\fg,K)$, we write $\AV(X) \in (\fg/\fk)^*$ for $\varphi[\gr(X)]$.

The next result relates the associated variety of $X$ to the associated variety of its annihilator.
\begin{theorem}\label{thm:AV}
Suppose $X$ is an irreducible $(\fg,K)$-module. Let
$$I := \mathrm{Ann}(X) \subset U(\fg),$$
a primitive ideal in $U(\fg)$. Then the following are true:

\begin{itemize}
    \item[(i)] $\AV(I)$ is the closure of a nilpotent orbit $\OO_G \subset \cN$.
    \item[(ii)] $\OO_G \cap (\fg/\fk)^*$ is the union of finitely-many $K$-orbits,
    $$\OO_G \cap (\fg/\fk)^* = \OO_1 \cup ... \cup \OO_t,$$
    and each $\OO_i$ a Lagrangian subvariety of $\OO_G$.
    \item[(iii)] Some $\OO_i$ are contained in $\AV(X)$, and they are the maximal $K$-orbits therein.
    \item[(iv)]  If $\AV(X)$ is reducible, then 
    $$\codim(\partial \OO_G,\overline{\OO}_G) =2.$$
\end{itemize}
\end{theorem}

\begin{proof}
(i) follows from the main theorem in \cite{Joseph1985}. (ii) and (iii) are \cite[Thm 8.4]{Vogan1991}. (iv) is \cite[Thm 4.6]{Vogan1991}.
\end{proof}

In Section \ref{sec:quantloc}, we will give a simple proof of Theorem \ref{thm:AV}(iv) using the machinery of microlocalization. 

In \cite{Vogan1991}, Vogan introduces a refinement of $\AV(X)$ called the \emph{associated $K$-cycle} which carries additional information about the $K$-action on $X$. Consider the variety
$$\cN_{\fk} := \cN \cap (\fg/\fk)^*.$$
The group $K$ acts on $\cN_{\fk}$ and, by (ii) of Theorem \ref{thm:AV}, with finitely many $K$-orbits. Denote the $K$-orbits on $\cN_{\fk}$ by
$$\cN_k = \OO_1 \cup ... \cup \OO_n.$$
Note that if $X$ is a finite-length Harish-Chandra module, then by (iii) of Theorem \ref{thm:AV}
$$\AV(X) \subseteq \cN_{\fk}.$$
\begin{definition}
An \emph{associated $K$-cycle} is an $n$-tuple
$$([\mathcal{E}_1],...,[\mathcal{E}_n]) \in K\mathrm{Vec}^K(\OO_1) \times ... \times K\mathrm{Vec}^K(\OO_n)$$
subject to the following two requirements
\begin{enumerate}
\item All of the classes $[\mathcal{E}_i]$ are genuine, i.e. they are represented by objects in $\mathrm{Vec}^K(\OO_i)$.
\item The $K$-orbits corresponding to nonzero classes are mutually incomparable, i.e. none is bigger than another.
\end{enumerate}
\end{definition}
The set $S$ of of associated $K$-cycles forms a semigroup, with addition defined by
$$([\mathcal{E}_i]) + ([\mathcal{E}'_i]) = ([\mathcal{E}''_i]), \qquad [\mathcal{E}''_i]:= \begin{cases}
                               0 & \text{if $[\mathcal{E}_j], [\mathcal{E}_j'] \neq 0$ for $\OO_i \subset \overline{\OO}_j$} \\
                                    [\mathcal{E}_i] + [\mathcal{E}'_i]  & \text{else} \\
  \end{cases}$$
Now if $M \in \Coh^K(\mathcal{N}_{\fk})$, we can define an associated $K$-cycle 
\begin{equation}\label{eq:defAC}\varphi(M) = ([\varphi_1(M)],...,[\varphi_n(M)]),\end{equation}
where
$$\varphi_i(M) =
  \begin{cases}
                                   M|_{\OO_i} & \text{if $\OO_i$ is open in $\mathrm{Supp}(M)$} \\
                                   0 & \text{else} \\
  \end{cases}$$

The function $\varphi: \Coh^K(\mathcal{N}_{\fk}) \to S$ can be extended to the category $\Coh^K_{\mathcal{N}_{\fk}}(\mathfrak{g}/\mathfrak{k})^*$ of $K$-equivariant coherent sheaves set-theoretically supported in $\cN_{\fk}$ in the following manner: if $M \in \Coh^K_{\mathcal{N}_{\fk}}(\mathfrak{g}/\mathfrak{k})^*$ choose a finite filtration $0=M_t \subset M_{t-1}... \subset M_0 = M$ by $K$-equivariant subsheaves such that $M_i/M_{i+1} \in \Coh^K(\mathcal{N}_{\fk})$ for every $i$ (for example, take $M_i:= I(\mathcal{N}_{\fk})^iM$ for each $i$). Now define 
$$\varphi: \Coh^K_{\cN_{\fk}}(\mathfrak{g}/\mathfrak{k})^* \to S, \qquad \varphi(M) = \sum_{i=0}^{t-1}\varphi(M_i/M_{i+1}).$$
In \cite[Thm 2.13]{Vogan1991}, Vogan shows that $\varphi$ is well-defined and additive. If $X$ is a finite-length Harish-Chandra module, then  $[\gr(X)] \in K\Coh^K_{\cN_{\fk}}(\fg/\fk)^*$. So $\varphi$ induces an (additive) function on finite-length Harish-Chandra modules $\mathrm{AC}(X) = \varphi[\gr(X)]$, called the associated $K$-cycle.

If $X$ is a unipotent Harish-Chandra module, see Definition \ref{def:specialunipotent}, then $\mathrm{AC}(X)$ is of a very special form. Let $\OO \subset \cN_{\fk}$ be a $K$-orbit. Let $p:\widetilde{\OO} \to \OO$ be the universal $K$-equivariant cover and let $\omega_{\widetilde{\OO}}$ be the canonical bundle on $\widetilde{\OO}$ (i.e. the line bundle of top-degree differential forms).

\begin{definition}\label{def:admissibility2}
A $K$-equivariant vector bundle $\mathcal{E}$ on $\OO$ is \emph{admissible} if there is an isomorphism of $K$-equivariant vector bundles on $\widetilde{\OO}$
$$p^*\mathcal{E} \otimes p^*\mathcal{E} \simeq \underbrace{\omega_{\widetilde{\OO}} \oplus ... \oplus \omega_{\widetilde{\OO}}}_{N \text{ times}},$$
for some $N \in \ZZ_{>0}$.
\end{definition}

This condition has a very simple Lie-theoretic interpretation. If we choose a point $e \in \OO$, there is an equivalence of categories
$$\mathrm{Vec}^K(\OO) \simeq \mathrm{Rep}(K_e),$$
From left to right, the equivalence is given by restriction to the fiber over $e$. A vector bundle $\mathcal{E}$ is admissible if and only if the corresponding $K_e$-representation $\rho: K_e \to \mathrm{GL}(V)$ satisfies
\begin{equation}\label{eq:admissibility}2d\rho = \mathrm{Tr}|_{(\mathfrak{k}/\mathfrak{k}_e)^*}  \cdot \mathrm{Id}_V.\end{equation}
This is the `admissibility' condition described in \cite[Sec 7]{Vogan1991}. We note that if $0 \to \mathcal{E}_1 \to \mathcal{E}_2 \to \mathcal{E}_3 \to 0$ is a short exact sequence of $K$-equivariant vector bundles on $\OO$ and $\mathcal{E}_1$, $\mathcal{E}_3$ are admissible, then $\mathcal{E}_2$ is admissible (this is evident from (\ref{eq:admissibility})). So admissibility is a property which can be ascribed to classes in $K\mathrm{Vec}^K(\OO)$.

\begin{theorem}[Thm 8.7, \cite{Vogan1991}]\label{thm:unipotentadmissible}
Suppose $X$ is a unipotent $(\mathfrak{g},K)$-module (cf. Definition \ref{def:specialunipotent}) with associated $K$-cycle $\mathrm{AC}(X) = ([\mathcal{E}_1],...,[\mathcal{E}_n])$. Then every nonzero class $[\mathcal{E}_i]$ is admissible. 
\end{theorem}

Only certain nilpotent $K$-orbits are admissible, so Theorem \ref{thm:unipotentadmissible} imposes strong additional constraints on the associated varieties of unipotent representations.

\subsection{Vogan's conjecture}\label{subsec:Vogansconjecture}

In \cite[Sec 12]{Vogan1991}, Vogan formulates a conjecture regarding the $K$-structure of unipotent Harish-Chandra modules, under some conditions. 

\begin{conj}[Conj 12.1, \cite{Vogan1991}]\label{conj:Vogan}
Let $\OO_G \subset \cN$ be a nilpotent orbit such that
$$\codim(\partial \OO_G, \overline{\OO}_G) \geq 4,$$
and suppose $X$ is a unipotent Harish-Chandra module (cf. Definition \ref{def:specialunipotent}) such that 
$$\AV(\mathrm{Ann}(X)) = \overline{\OO}_G.$$
Then by Theorem \ref{thm:AV}(iv), 
$$\AV(X) = \overline{\OO},$$
for some $K$-orbit in $\OO_G \cap (\fg/\fk)^*$ (which is a Lagrangian subvariety of $\OO_G$) and by Theorem \ref{thm:unipotentadmissible}
$$\mathrm{AC}(X) = ([\mathcal{E}]),$$
where $\mathcal{E}$ is an admissible $K$-equivariant vector bundle on $\OO$. There is an isomorphism of algebraic $K$-representations
\begin{equation}\label{eq:VoganKtypes}X \simeq_K \Gamma(\OO, \mathcal{E})\end{equation}
\end{conj}

In Section \ref{sec:cohomology}, we will show that this conjecture is true if $\mathcal{E}$ satisfies a certain cohomological condition, which we can verify in many cases.

\subsection{Localization of categories}\label{subsec:localization}

In this section, we will review some basic aspects of the theory of localization of abelian categories. The discussion here follows \cite[Chapter 4]{Popescu1973}.

Let $\mathcal{C}$ be an abelian category. A full subcategory $\mathcal{B} \subset \mathcal{C}$ is \emph{Serre} if for every short exact sequence $0 \to X \to Y \to Z \to 0$ in $\mathcal{C}$,
$$Y \in \mathcal{B} \iff X \in \mathcal{B} \ \text{and} \ Z \in \mathcal{B}.$$
In other words, $\mathcal{B}$ is a full subcategory which is closed under the formation of subobjects, quotients, and extensions. 
\begin{example}\label{ex:restriction1}
Let $X$ be a variety and $Z \subset X$ a closed subvariety. Then $\Coh_Z(X)$ is a Serre subcategory of $\Coh(X)$.
\end{example}
Given a Serre subcategory $\mathcal{B} \subset \mathcal{C}$, one can define the \emph{quotient category} $\mathcal{C}/\mathcal{B}$. The next (very standard) result describes the universal property which it satisfies.

\begin{prop}[Cor 3.11, \cite{Popescu1973}]\label{prop:quotcat}
Let $\mathcal{C}$ be an abelian category and $\mathcal{B} \subset \mathcal{C}$ a Serre subcategory. There is an abelian category $\mathcal{C}/\mathcal{B}$, unique up to equivalence, receiving an exact, essentially surjective functor $T: \mathcal{C} \to \mathcal{C}/\mathcal{B}$ with
$$\mathcal{B}=\ker{T} := \{C \in \mathcal{C}: TC=0\}$$
satisfying the following universal property: if $F: \mathcal{A} \to \mathcal{D}$ is an exact functor with $\mathcal{B} \subseteq \ker{F}$, then there is a unique exact functor $G: \mathcal{C}/\mathcal{B} \to \mathcal{D}$ such that $F = G \circ T$. 
\end{prop}

\begin{example}\label{ex:restriction2}
Let $X$ be a variety, $Z \subset X$ a closed subvariety with open complement $U$. Write $j: U \subset X$ for the inclusion. Recall, Example \ref{ex:restriction1}, that $\Coh_Z(X) \subset \Coh(X)$ is a Serre subcategory. Restriction to $U$ defines an exact essentially surjective functor $j^*:\Coh(X) \to \Coh(U)$ with kernel $\Coh_Z(X)$. It is not difficult to show that $j^*$ satisfies the universal property in Proposition \ref{prop:quotcat}. Hence, 
$\Coh(X)/\Coh_Z(X) \xrightarrow{\sim} \Coh(U)$. 
\end{example}

We are interested in Serre subcategories of a very special type. A Serre subcategory $\mathcal{B}$ is \emph{localizing} if the quotient functor $T: \mathcal{C} \to \mathcal{C}/\mathcal{B}$ admits a right-adjoint $L: \mathcal{C}/\mathcal{B}\to \mathcal{C}$. We call the composition $LT: \mathcal{C} \to \mathcal{C}$ the \emph{localization} of $\mathcal{C}$ with respect to the localizing subcategory $\mathcal{B}$. 

\begin{example}
Return to the setting of Example \ref{ex:restriction2}. First, assume that $\codim(Z,X) \geq 2$. Then $j^*:\Coh(X) \to \Coh(U)$ admits a right adjoint (and right inverse), namely $j_*: \Coh(U) \to \Coh(X)$. Hence, $\Coh_Z(X)$ is a localizing subcategory of $\Coh(X)$ and $j_*j^*: \Coh(X) \to \Coh(X)$ is the localization of $\Coh(X)$ with respect to $\Coh_Z(X)$. If, on the other hand, $\codim(Z,X)=1$, the direct image under $j$ of a coherent sheaf on $U$ is an object in $\mathrm{QCoh}(X)$, though not of $\Coh(X)$. In fact, $j^*$ in this case does not admit a right adjoint, i.e. $\Coh_Z(X)$ is not a localizing subcategory of $\Coh(X)$.
\end{example}

The following proposition catalogs the essential properties of $L$. 

\begin{prop}[Prop 4.3, \cite{Popescu1973}]\label{prop:generalpropsoflocalization}
Suppose $\mathcal{B}$ is a localizing subcategory of an abelian category $\mathcal{C}$ and that $L: \mathcal{C}/\mathcal{B}\to \mathcal{C}$ is right adjoint to the quotient functor $T: \mathcal{C} \to \mathcal{C}/\mathcal{B}$.
\begin{enumerate}
\item $L$ is left exact
\item $TL$ is naturally isomorphic to $\mathrm{id}_{\mathcal{C}/\mathcal{B}}$
\item An object $C \in \mathcal{C}$ is in the image of $L$ if and only if it has no nontrivial maps from, or extensions by, objects in $\mathcal{B}$. Symbolically, 
$$C \in \mathrm{Im}(L) \iff \mathrm{Hom}(\mathcal{B},C)=\mathrm{Ext}^1(C,\mathcal{B})=0$$
\item For every object $C \in \mathcal{C}$, the canonical morphism $C \to LT(C)$ has kernel and cokernel in $\mathcal{B}$.
\end{enumerate}
\end{prop}

We conclude this section with a useful criterion which we will use below.

\begin{prop}[Thm 4.9, \cite{Popescu1973}]\label{prop:criterionquotient}
Suppose $T:\mathcal{C} \to \mathcal{A}$ is an exact functor of abelian categories with a fully faithful right adjoint $L: \mathcal{A} \to \mathcal{C}$.  Then $T$ is a quotient functor and $\mathcal{A} \simeq \mathcal{C}/\ker{T}$.
\end{prop}

\section{Microlocalization of Harish-Chandra modules}\label{sec:microlocalization}

\subsection{The Rees construction}\label{sec:Rees}

We want to perform operations on filtered Harish-Chandra modules analogous to the restriction and extension of coherent sheaves on $\cN_{\fk}$. The first problem we encounter is that the category $\mathrm{HC}^{\mathrm{filt}}(\mathfrak{g},K)$ of well-filtered Harish-Chandra modules is not abelian. Cokernels are not well-defined. The solution is to pass to a larger abelian category containing $\mathrm{HC}^{\mathrm{filt}}(\mathrm{g},K)$.

Let $A$ be an associative algebra equipped with an increasing filtration by subspaces 
$$...\subseteq 
A_{-1} \subseteq A_0 \subseteq A_1 \subseteq ..., \qquad A_mA_n \subseteq A_{m+n}, \qquad \bigcap_m A_m=0, \qquad \bigcup_m A_m = A.$$
Form the polynomial algebra $A[\hbar,\hbar^{-1}]$ in the formal symbol $\hbar$. Define a $\mathbb{Z}$-grading by declaring $\mathrm{deg}(A) = 0$ and $\mathrm{deg}(\hbar) =1$. The \emph{Rees algebra} of $A$ is the graded subalgebra
$$R_{\hbar}A = \bigoplus_{m \in \mathbb{Z}} A_m\hbar^m \subset A[\hbar,\hbar^{-1}].$$
In a precise sense, $R_{\hbar}A$ interpolates between $A$ and $\gr(A)$. 

\begin{prop}\label{prop:reesfacts}
The subspaces $\hbar R_{\hbar}A \subset R_{\hbar}A$ and $(\hbar-1)R_{\hbar}A \subset R_{\hbar}A$ are two-sided ideals and
\begin{enumerate}
\item There is a canonical isomorphism of graded algebras 
$$R_{\hbar}A/\hbar R_{\hbar}A \simeq \gr(A).$$
\item There is a canonical isomorphism of filtered algebras
$$R_{\hbar}A/(\hbar -1)R_{\hbar}A \simeq A.$$
\end{enumerate}
\end{prop}

\begin{proof}
The ideals are two-sided since the elements $\hbar, \hbar-1$ are central.

\begin{enumerate}
\item The linear maps $A_n\hbar^n \simeq A_n \to A_n/A_{n-1}$ assemble into a surjective homomorphism of graded algebras
$$R_{\hbar}A \to \gr(A).$$
The kernel of this homomorphism is the graded subalgebra 
$$\bigoplus_nA_{n-1}\hbar^n = \bigoplus_nA_n\hbar^{n+1} = \hbar R_{\hbar}A.$$

\item The inclusions $A_n\hbar^n \simeq A_n \subseteq A$ assemble into a filtered homomorphism
$$
i: R_{\hbar}A \to A
$$
which is surjective since the filtration is exhaustive. Choose an element $a$ in the kernel of $i$

$$a = a_p\hbar^p + a_{p+1}\hbar^{p+1} + ... + a_q \hbar^q, \qquad \sum_{n=p}^q a_n = 0.$$
If we define $b_n = -a_p - a_{p+1} - ... - a_n \in A_n$ for $p \leq n \leq q$, then one easily computes
$$
(\hbar -1)\sum_{n=p}^{q-1}b_n\hbar^n = a.
$$
In particular, $a \in (\hbar - 1)R_{\hbar}A$. On the other hand, the subalgebra $(\hbar - 1)R_{\hbar}A$ is spanned by the elements $(\hbar - 1)a_n\hbar^n$ and by an easy computation $i((\hbar-1)a_n\hbar^n) =0$.
\end{enumerate}
\end{proof}

Now suppose $M$ is a module for $A$ equipped with an increasing filtration by subspaces compatible with the filtration on $A$
$$... \subseteq M_{-1} \subseteq M_0 \subseteq M_1 \subseteq ..., \qquad A_mM_n \subseteq M_{m+n}, \qquad \bigcap_m M_m =0, \qquad \bigcup_m M_m = M. $$
Form the vector space $M[\hbar,\hbar^{-1}]$ and define a $\mathbb{Z}$-grading (as above) by $\deg(M) = 0$ and $\deg(\hbar)=1$. Now $M[\hbar,\hbar^{-1}]$ is a graded module for $R_{\hbar}A$ due to the compatibility of the filtration. The \emph{Rees module} of $M$ is the graded $R_{\hbar}A$-submodule

$$M_{\hbar} = \bigoplus_{m \in \mathbb{Z}}M_m\hbar^m \subset M[\hbar,\hbar^{-1}].$$

Take $A = U(\mathfrak{g})$ with its standard filtration. The group $K$ acts on $U(\mathfrak{g})$ by filtered algebra automorphisms and therefore on its Rees algebra $R_{\hbar}U(\mathfrak{g})$ by graded algebra automorphisms. A $(\mathfrak{g}_{\hbar},K)$-module is a graded left $R_{\hbar}U(\mathfrak{g})$-module $X_{\hbar}$ equipped with a graded algebraic $K$-action satisfying the following two conditions:
\begin{enumerate}
\item The action map $R_{\hbar}U(\mathfrak{g}) \otimes X_{\hbar} \to X_{\hbar}$ is $K$-equivariant,
\item The $R_{\hbar}U(\mathfrak{g})$-action, restricted to the subspace $\mathfrak{k}\hbar \subset \mathfrak{g}\hbar \subset R_{\hbar}U(\mathfrak{g})$, coincides with $\hbar$ times the differentiated action of $K$. 
\end{enumerate}
A morphism of $(\mathfrak{g}_{\hbar},K)$-modules is a graded homomorphism of $R_{\hbar}U(\mathfrak{g})$-modules intertwining the $K$-actions. Write $M(\mathfrak{g}_{\hbar},K)$ for the abelian category of $(\mathfrak{g}_{\hbar},K)$-modules and $\mathrm{HC}(\mathfrak{g}_{\hbar},K)$ for the full subcategory of finitely-generated $(\mathfrak{g}_{\hbar},K)$-modules. 

The assignment $X \mapsto R_{\hbar}X$ defines a functor from the category $\mathrm{HC}^{\mathrm{filt}}(\mathfrak{g},K)$ of well-filtered Harish-Chandra modules to $\mathrm{HC}(\mathfrak{g}_{\hbar},K)$.

\begin{prop}
If $X \in \mathrm{HC}^{\mathrm{filt}}(\mathfrak{g},K)$ (with filtration $...\subseteq X_{-1} \subseteq X_0 \subseteq X_1 \subseteq ...$), $R_{\hbar}X$ has the structure of a $(\mathfrak{g}_{\hbar},K)$-module, finitely-generated over $R_{\hbar}U(\mathfrak{g})$. The assignment $X \mapsto R_{\hbar}X$ defines a functor
$$R_{\hbar}: \mathrm{HC}^{\mathrm{filt}}(\mathfrak{g},K) \to \mathrm{HC}(\mathfrak{g}_{\hbar},K)$$
defined on morphisms $f: X \to Y$ by $(R_{\hbar}f)(x\hbar^m) = f(x)\hbar^m$. This functor is a fully-faithful embedding. Its image is the subcategory $\mathrm{HC}^{\mathrm{tf}}(\mathfrak{g}_{\hbar},K)$ of Harish-Chandra $(\mathfrak{g}_{\hbar},K)$-modules which are $\hbar$-torsion-free.
\end{prop}

\begin{proof}
There is a functor 
$$\hbar=1: \mathrm{HC}(\mathfrak{g}_{\hbar},K) \to \mathrm{HC}^{\mathrm{filt}}(\mathfrak{g},K)$$
defined by $X_{\hbar} \mapsto X_{\hbar}/(\hbar-1)X_{\hbar}$.  The argument provided in the proof of Proposition \ref{prop:reesfacts} (replacing rings with modules) shows that $(\hbar=1) \circ R_{\hbar}$ is the identity functor on $\mathrm{HC}^{\mathrm{filt}}(\mathfrak{g},K)$. It remains to exhibit a natural isomorphism $R_{\hbar}(X_{\hbar}/(\hbar-1)X_{\hbar}) \simeq X_{\hbar}$ for every $X_{\hbar} \in \mathrm{HC}^{\mathrm{tf}}(\mathfrak{g}_{\hbar},K)$.

Fix $X_{\hbar} \in \mathrm{HC}^{\mathrm{tf}}(\mathfrak{g}_{\hbar},K)$ and write $X_{\hbar}^n$ for its $n$th graded component. For every integer $N$, define the graded subspace
$$X_{\hbar}^{\leq N} = \bigoplus_{n \leq N} X_{\hbar}^n.$$
There is a linear map
$$
\varphi_N: X_{\hbar}^{\leq N} \to X_{\hbar}^N, \qquad \varphi_N(x) = \sum_{n \leq N} x^n\hbar^{N-n}.
$$
This map is surjective, since (for example) it restricts to the identity map on $X_{\hbar}^N$. We will show that
$$\ker{\varphi_N}=(\hbar-1)X_{\hbar} \cap X_{\hbar}^{\leq N}.$$
Suppose 
$$(\hbar-1)(x^p + x^{p+1} + ... + x^q) \in (\hbar-1)X_{\hbar} \cap X_{\hbar}^{\leq N}.$$
Then $\hbar x^n - x^{n+1} = 0$ for every $n \geq N$ and consequently $x \in X_{\hbar}^{\leq N-1}$ since $X_{\hbar}$ is $\hbar$-torsion free. Then by a simple computation  $\varphi_N((\hbar-1)(x^p+...+x^q))=0$.

Conversely, suppose 
$$x = x^p + x^{p+1} + ... + x^N \in \ker{\varphi_N}.$$
Then $\sum_{n \leq N}x^n \hbar^{N-n} = 0$. For $n \leq N$, define
$$
y^n = -x^n - \hbar x^{n-1} - \hbar^2 x^{n-2} - ... \in X_{\hbar}^n.
$$
Then by a simple computation
$$x=(\hbar-1)(y^{N-1} + y^{N-2} + ...) \in (\hbar-1)X_{\hbar} \cap X_{\hbar}^{\leq N}.$$
This proves $\ker{\varphi_N}=(\hbar-1)X_{\hbar} \cap X_{\hbar}^{\leq N}$. As a result, $\varphi_N$ induces a linear isomorphism
$$
\varphi_N: (X_{\hbar}/(\hbar-1)X_{\hbar})^{\leq N} = X_{\hbar}^{\leq N}/\left((\hbar-1)X_{\hbar}\cap X_{\hbar}^{\leq N}\right) \simeq X_{\hbar}^N.
$$
We can assemble these maps into a graded isomorphism
$$
\varphi = \bigoplus_N \varphi_N: R_{\hbar}(X_{\hbar}/(\hbar-1)X_{\hbar}) \simeq X_{\hbar}.
$$
It is clear from its construction that $\varphi$ is a $R_{\hbar}U(\mathfrak{g})$-module homomorphism and is compatible with the $K$-actions. 
\end{proof}

Besides $R_{\hbar}$, there are several other functors relating the categories $\mathrm{HC}^{\mathrm{filt}}(\mathfrak{g},K)$, $\mathrm{HC}(\mathfrak{g}_{\hbar},K)$ and $\Coh^{K \times \CC^{\times}}(\mathfrak{g}/\mathfrak{k})^*$. By Proposition \ref{prop:reesfacts} applied to $A=U(\mathfrak{g})$ there are natural isomorphisms
$$R_{\hbar}U(\mathfrak{g})/\hbar R_{\hbar}U(\mathfrak{g}) \simeq S(\mathfrak{g}), \qquad R_{\hbar}U(\mathfrak{g})/(\hbar-1) R_{\hbar}U(\mathfrak{g}) \simeq U(\mathfrak{g}).$$
Every $M \in \Coh^{K \times \CC^{\times}}(\mathfrak{g}/\mathfrak{k})^*$ can be regarded as a finitely-generated $(\mathfrak{g}_{\hbar},K)$-module via the quotient map $R_{\hbar}U(\mathfrak{g}) \to S(\mathfrak{g})$. On the other hand, if $X_{\hbar} \in \mathrm{HC}^{\mathrm{filt}}(\mathfrak{g}_{\hbar},K)$, then $X_{\hbar}/\hbar X_{\hbar}$ has the structure of a graded, $K$-equivariant, coherent sheaf on $(\mathfrak{g}/\mathfrak{k})^*$ and $X_{\hbar}/(\hbar-1)X_{\hbar}$ has the structure of a well-filtered Harish-Chandra module. These operations define functors, which are related as described in the following proposition. 

\begin{prop}\label{prop:functors}
The functors
\begin{align*}
i: \Coh^{K \times \CC^{\times}}(\mathfrak{g}/\mathfrak{k})^* &\to \mathrm{HC}(\mathfrak{g}_{\hbar},K)\\
M &\mapsto M\\
\hbar=0: \mathrm{HC}(\mathfrak{g}_{\hbar},K) &\to \Coh^K(\mathfrak{g}/\mathfrak{k})^*\\
X_{\hbar} &\mapsto X_{\hbar}/\hbar X_{\hbar}\\
\hbar=1: \mathrm{HC}(\mathfrak{g}_{\hbar},K) &\to \mathrm{HC}^{\mathrm{filt}}(\mathfrak{g},K)\\
X_{\hbar} &\mapsto X_{\hbar}/(\hbar-1)X_{\hbar}\\
\gr: \mathrm{HC}^{\mathrm{filt}}(\mathfrak{g},K) &\to \Coh^{K \times \CC^{\times}}(\mathfrak{g}/\mathfrak{k})^*\\
X &\mapsto \gr(X)
\end{align*}
satisfy the relations
\begin{enumerate}
\item $(\hbar=0) \circ i = \mathrm{id}$
\item $(\hbar=1) \circ R_{\hbar}=\mathrm{id}$
\item $(\hbar=0) \circ R_{\hbar} = \gr$
\end{enumerate}
\end{prop}

\begin{proof}
The first relation is obvious. The second and third follow from Proposition \ref{prop:reesfacts} (replacing $A$ with $X$).
\end{proof}

The situation is summarized in the following commutative diagram.

\begin{center}
\begin{tikzcd}[column sep = large, row sep = large]
& \mathrm{HC}(\mathfrak{g}_{\hbar},K) \arrow[dl, "\hbar=0", shift left=.5ex] \arrow[r, hookrightarrow] \arrow[d,"\hbar=1", shift left=.5ex]& M(\mathfrak{g}_{\hbar},K)\\
\Coh^{K \times \CC^{\times}}(\mathfrak{g}/\mathfrak{k})^* \arrow[ur, hookrightarrow, "i", shift left=.5ex] & \mathrm{HC}^{\mathrm{filt}}(\mathfrak{g},K) \arrow[u,hookrightarrow, "R_{\hbar}", shift left=.5ex] \arrow[l, "\gr"] \arrow[d, "\text{forget}"] & \\
& \mathrm{HC}(\mathfrak{g},K) \arrow[r,hookrightarrow] & M(\mathfrak{g},K)
\end{tikzcd}
\end{center}

\subsection{Construction of $\Phi_{\OO}$}\label{sec:quantloc}

Choose $e \in \cN_{\fk}$ and let $\OO = K \cdot e \subset \cN_{\fk}$. Recall that if $X_{\hbar} \in M(\mathfrak{g}_{\hbar},K)$, then $X_{\hbar}/\hbar X_{\hbar} \in \mathrm{QCoh}^{K \times \CC^{\times}}(\fg/\fk)^*$. Define the \emph{associated variety} of $X_{\hbar}$ to be the subset
$$\AV(X_{\hbar}) := \mathrm{Supp}(X_{\hbar}/\hbar X_{\hbar}) \subset (\fg/\fk)^*.$$
If $Z$ is a subset of $(\mathfrak{g}/\mathfrak{k})^*$, we can define full subcategories $\QCoh_Z^{K,\mathbb{C}^{\times}}(\mathfrak{g}/\mathfrak{k})^*$, $\Coh_Z^{K,\mathbb{C}^{\times}}(\mathfrak{g}/\mathfrak{k})^*$, $M_Z(\mathfrak{g}_{\hbar},K)$, $\mathrm{HC}_Z(\mathfrak{g}_{\hbar},K)$, $\mathrm{HC}_Z^{\mathrm{filt}}(\mathfrak{g},K)$, and $\mathrm{HC}_Z(\mathfrak{g},K)$ by considering all objects with set-theoretic support (or associated variety) contained in $Z$. We begin with a simple observation.

\begin{prop}\label{prop:Serresubcat}
The subcategories 
\begin{align*}
\QCoh^{K \times \CC^{\times}}_{\partial \OO}(\mathfrak{g}/\mathfrak{k})^* &\subset \QCoh^{K \times \CC^{\times}}_{\overline{\OO}}(\mathfrak{g}/\mathfrak{k})^*\\
\Coh^{K \times \CC^{\times}}_{\partial \OO}(\mathfrak{g}/\mathfrak{k})^* &\subset \Coh^{K \times \CC^{\times}}_{\overline{\OO}}(\mathfrak{g}/\mathfrak{k})^*\\
\mathrm{HC}_{\partial \OO}(\mathfrak{g},K) &\subset \mathrm{HC}_{\overline{\OO}}(\mathfrak{g},K)\\
\mathrm{HC}_{\partial \OO}(\mathfrak{g}_{\hbar},K) &\subset \mathrm{HC}_{\overline{\OO}}(\mathfrak{g}_{\hbar},K)
\end{align*}
are Serre.
\end{prop}

\begin{proof}
The first two subcategories are Serre by the additivity of support. The third subcategory is Serre by the additivity of the associated variety (see the remarks preceding Theorem \ref{thm:AV}). For the fourth subcategory, we argue as follows. Suppose $A$ is filtered algebra with an algebraic $K$-action such that $\gr(A)$ is a finitely-generated commutative algebra. One can define Harish-Chandra $(A,K)$-modules and good filtrations on them as in Section \ref{subsec:associatedstuff}. Similarly to Proposition \ref{prop:grprop}, every Harish-Chandra $(A,K)$-module admits a good filtration and there is a group homomorphism
$$K\mathrm{HC}(A,K) \to K\Coh^K(\mathrm{Spec}(\gr(A)))$$
The algebra $R_{\hbar}U(\mathfrak{g})$ has two natural $K$-invariant filtrations. One is the filtration defined by the grading. The second is defined by
$$(R_{\hbar}U(\mathfrak{g}))_m = U_m(\mathfrak{g})[\hbar] \cap R_{\hbar}U(\mathfrak{g}) = \hbar^m U_m(\mathfrak{g})[\hbar]$$
If we equip $R_{\hbar}U(\fg)$ with the latter filtration, there is a natural identification $\gr R_{\hbar} U(\fg) \simeq S(\mathfrak{g})[\hbar]$. So if $X_{\hbar} \in \mathrm{HC}(\mathfrak{g}_{\hbar},K)$, then we can find a good filtration on $X_{\hbar}$ (compatible with our chosen filtration on $R_{\hbar}U(\fg)$)  and $[\gr X_{\hbar}] \in K\Coh^K((\mathfrak{g}/\fk)^* \times \mathrm{Spec}(\CC[\hbar]))$. It is easy to see that 
$$\AV(X_{\hbar}) = \mathrm{Supp}(\gr X_{\hbar}/\hbar \gr X_{\hbar}).$$
The proof is very easy and we omit it. Now it follows that $\mathrm{HC}_{\partial \OO}(\mathfrak{g}_{\hbar},K) \subset \mathrm{HC}_{\overline{\OO}}(\mathfrak{g}_{\hbar},K)$ is Serre. 
\end{proof}

\begin{rmk}
Note that the subcategory $M_{\partial \OO}(\mathfrak{g}_{\hbar},K) \subset M_{\overline{\OO}}(\mathfrak{g}_{\hbar},K)$ is not Serre. It is closed under quotients and extensions, but not under subobjects. Take, for instance, $\mathfrak{g} = \mathbb{C}$. Then $R_{\hbar}U(\mathfrak{g}) = \mathbb{C}[x,\hbar]$. Let $M= \mathbb{C}(\hbar)$, the field of rational functions in $\hbar$. $M$ is an (infinitely-generated) $\mathbb{C}[x,\hbar]$-module with $x$ acting by $0$ and $M/\hbar M=0$ (since $\mathbb{C}(\hbar)$ is a field). Hence, $\mathrm{Supp}(M) = \emptyset$. Yet the submodule $L=\mathbb{C}[\hbar]$ has $L/\hbar L=\mathbb{C}$ and is therefore supported at a point.
\end{rmk}

Our goal in this section is to prove (under a codimension condition on $\partial \OO$) that $\mathrm{HC}_{\partial \OO}(\mathfrak{g}_{\hbar},K)$ is a localizing subcategory of $\mathrm{HC}_{\overline{\OO}}(\mathfrak{g}_{\hbar},K)$ and to construct the corresponding localization functor
$$\Phi_{\OO}: \mathrm{HC}_{\overline{\OO}}(\mathfrak{g}_{\hbar},K) \to \mathrm{HC}_{\overline{\OO}}(\mathfrak{g}_{\hbar},K),$$
see Section \ref{subsec:localization}. We will see that $\Phi_{\OO}$ descends to a functor
$$\overline{\Phi}_{\OO}: \mathrm{HC}_{\overline{\OO}}(\mathfrak{g},K) \to \mathrm{HC}_{\overline{\OO}}(\mathfrak{g},K)$$
which inherits all of the essential properties of $\Phi_{\OO}$.

Our construction of $\Phi_{\OO}$ is adapted from Losev, who constructs analogous functors in \cite[Sec 3]{Losev2011} for Harish-Chandra bimodules. Most of the proofs in this section are due essentially to Losev, although some arguments have been modified to accommodate our more general setting.

Fix an element $e \in \OO$. If we fix an invariant form on $\mathfrak{g}$, $e$ is identified with a nilpotent element in $\mathfrak{p}$, which we continue to denote by $e$. Choose an $\mathfrak{sl}(2)$ triple $(e,f,h) \in \mathfrak{p} \times \mathfrak{p} \times \mathfrak{k}$. The centralizer $L := K^{e,f,h}$ is a maximal reductive subgroup of $K^e$.

Define the maximal ideal $I_{e} \subset R_{\hbar}U(\mathfrak{g})$ to be the preimage under the canonical surjection $R_{\hbar}U(\mathfrak{g}) \to S(\mathfrak{g})$ of the maximal ideal defining $e$. Then consider the completion of $R_{\hbar}U(\mathfrak{g})$ with respect to $I_{e}$:
$$\widehat{R_{\hbar}U(\mathfrak{g})} := \varprojlim R_{\hbar}U(\mathfrak{g})/I_{e}^n.$$
There is a canonical surjection $\widehat{R_{\hbar}U(\mathfrak{g})} \to R_{\hbar}U(\mathfrak{g})/I_{e}$. Since $R_{\hbar}U(\mathfrak{g})/I_{e}$ is a field, the kernel $\widehat{I_{e}} \subset \widehat{R_{\hbar}U(\mathfrak{g})}$ is the unique maximal (left, right, and two-sided) ideal in $\widehat{R_{\hbar}U(\mathfrak{g})}$. 

The basic properties of the algebra $\widehat{R_{\hbar}U(\mathfrak{g})}$ and its finitely-generated modules are summarized in \cite{Losev2011}. Here are a few:

\begin{prop}[Lem 2.4.4, \cite{Losev2011}]\label{prop:completionfacts}
The following are true:
\begin{enumerate}
\item $\widehat{R_{\hbar}U(\mathfrak{g})}$ is Noetherian
\item $\widehat{R_{\hbar}U(\mathfrak{g})}$ is separated and complete in the $\widehat{I_{e}}$-adic topology, i.e. 
$$\bigcap_n \widehat{I_{e}}^n\widehat{R_{\hbar}U(\mathfrak{g})}=0, \qquad \widehat{R_{\hbar}U(\mathfrak{g})} \simeq \varprojlim \widehat{R_{\hbar}U(\mathfrak{g})}/\widehat{I_{e}}^n$$
\item If $\widehat{X}_{\hbar}$ is a finitely-generated $\widehat{R_{\hbar}U(\mathfrak{g})}$-module, $\widehat{X}_{\hbar}$ is separated and complete in the $\widehat{I_{e}}$-adic topology, i.e.
$$\bigcap_n \widehat{I_{e}}^n\widehat{X}_{\hbar}=0, \qquad \widehat{X}_{\hbar} \simeq \varprojlim \widehat{X}_{\hbar}/\widehat{I_{e}}^n\widehat{X}_{\hbar}$$
\end{enumerate}
\end{prop}

If a group (or Lie algebra) acts on $R_{\hbar}U(\mathfrak{g})$ and preserves $I_{e}$, then it acts naturally on the completion $\widehat{R_{\hbar}U(\mathfrak{g})}$. There are two reasonable group actions with this property:

\begin{enumerate}
\item \textbf{The adjoint action of} $L$. Since $L$ preserves $e \in \mathfrak{g}^*$, it preserves the maximal ideal defining it and therefore its preimage $I_{e} \subset R_{\hbar}U(\mathfrak{g})$. Consequently, it lifts to an action on $\widehat{R_{\hbar}U(\mathfrak{g})}$. In fact, the entire centralizer $K^{e}$ acts in this fashion, but for reasons that will soon become apparent we will not consider the action of the unipotent radical.

\item \textbf{The Kazhdan action of} $\mathbb{C}^{\times}$. The element $h \in \mathfrak{k}$ determines a unique co-character $\gamma: \mathbb{C}^{\times} \to K$ with $d\gamma_1(1) = h$. We get an algebraic action of $\mathbb{C}^{\times}$ on $U(\mathfrak{g})$ by composing $\gamma$ with $\mathrm{Ad}$:
$$t \cdot X_1...X_m = \mathrm{Ad}(\gamma(t))(X_1)...\mathrm{Ad}(\gamma(t))(X_m)$$
Finally, we extend this action to the polynomial algebra $U(\mathfrak{g})[\hbar]$ by defining $t \cdot \hbar = t^2\hbar$. This action obviously preserves the subalgebra $R_{\hbar}U(\mathfrak{g}) \subset U(\mathfrak{g})[\hbar]$. 

$\mathbb{C}^{\times}$ also acts on $\mathfrak{g}^*$ by $t \cdot \zeta = t^{-2} \mathrm{Ad}^*(\gamma(t))(\zeta)$. This induces a $\mathbb{C}^{\times}$-action on $S(\mathfrak{g}) = \mathbb{C}[\mathfrak{g}^*]$ characterized by $t \cdot X = t^2 \mathrm{Ad}(\gamma(t))(X)$. These actions (of $\mathbb{C}^{\times}$ on $R_{\hbar}U(\mathfrak{g}), S(\mathfrak{g})$, and $\mathfrak{g}^*$) are what Losev calls in \cite{Losev2011} the `Kazhdan actions' of $\mathbb{C}^{\times}$. The canonical map $R_{\hbar}U(\mathfrak{g}) \to S(\mathfrak{g})$ is equivariant with respect to the Kazhdan actions on $R_{\hbar}U(\mathfrak{g})$ and $S(\mathfrak{g})$. The definitions are rigged so that $e$ is fixed by $\mathbb{C}^{\times}$:
\begin{align*}
t \cdot e &= t^{-2} \gamma(t) \cdot e = t^{-2} e(\gamma(t)^{-1} \cdot) = t^{-2} (e, \gamma(t)^{-1}\cdot)\\
&= t^{-2}(\gamma(t)\cdot e, \cdot) = t^{-2}(t^2e,\cdot) = (e,\cdot) = e
\end{align*}
Hence, the Kazhdan action preserves the ideal defining $e$ and therefore its preimage $I_{e} \subset R_{\hbar}U(\mathfrak{g})$. Consequently, it lifts to an action on $\widehat{R_{\hbar}U(\mathfrak{g})}$.
\end{enumerate}

Two comments on these definitions are in order. First, since $L$ centralizes $\gamma(\mathbb{C}^{\times})$, these two actions commute. This is not the case if we consider the full action of $K^{e}$ and this is the principal reason why we restrict our attention to $L$. Second, neither action is algebraic (i.e. locally finite), except in the most trivial situations. However, both actions can be differentiated (to the Lie algebras $\mathfrak{l}$ and $\mathbb{C}$, respectively) since they are lifted from algebraic actions on $R_{\hbar}U(\mathfrak{g})$.

Suppose $X_{\hbar} \in M(\mathfrak{g}_{\hbar},K)$. Form the completion of $X_{\hbar}$ with respect to $I_{e}$:
$$\widehat{X}_{\hbar} := \varprojlim X_{\hbar}/I_{e}^nX_{\hbar}.$$
The space $\widehat{X}_{\hbar}$ has lots of interesting structure. For one, it is obviously a module for $\widehat{R_{\hbar}U(\mathfrak{g})}$. The $L$-action on $X_{\hbar}$ lifts to an action on $\widehat{X}_{\hbar}$, since $L$ preserves $I_{e}$. The naive $\mathbb{C}^{\times}$-action on $X_{\hbar}$ (obtained from the grading) \emph{does not} lift to the completion (since $I_{e}X_{\hbar}$ is not usually graded). But as with $R_{\hbar}U(\mathfrak{g})$ we can define a slightly modified action (called the \emph{Kazhdan action} on $X_{\hbar}$) by
$$t \cdot x = t^{2n}\gamma(t)x, \qquad n = \mathrm{deg}(n),$$
and this action \emph{does} lift to the completion. Therefore, $\widehat{X}_{\hbar}$ has the structure of a $\widehat{R_{\hbar}U(\mathfrak{g})}$-module with actions of $L$ and $\mathbb{C}^{\times}$. Once again, these actions are \emph{not} algebraic. But they \emph{do} differentiate to the Lie algebras. The axioms for $M(\mathfrak{g}_{\hbar},K)$ impose various compatibility conditions on these three algebraic structures.

\begin{prop}\label{prop:structureofcompletion}
If $X_{\hbar} \in M(\mathfrak{g}_{\hbar},K)$, then $\widehat{X}_{\hbar}$ has the structure of a $\widehat{R_{\hbar}U(\mathfrak{g})}$-module with actions of $L$ and $\mathbb{C}^{\times}$ satisfying the following properties

\begin{enumerate}
\item\label{structureofcompletion1} The $L$ and $\mathbb{C}^{\times}$-actions commute
\item\label{structureofcompletion2} The action map $\widehat{R_{\hbar}U(\mathfrak{g})} \otimes \widehat{X}_{\hbar} \to \widehat{X}_{\hbar}$ is both $L$ and $\mathbb{C}^{\times}$-equivariant
\item\label{structureofcompletion3} The $\widehat{R_{\hbar}U(\mathfrak{g})}$-action, restricted to the subspace $\mathfrak{l}\hbar \subset \mathfrak{g}\hbar \subset R_{\hbar}U(\mathfrak{g}) \subset \widehat{R_{\hbar}U(\mathfrak{g})}$ coincides with $\hbar$ times the differentiated action of $L$.
\end{enumerate}
\end{prop}

A $(\widehat{g}_{\hbar},L)$-module is a left $\widehat{R_{\hbar}U(\mathfrak{g})}$-module with $L$ and $\mathbb{C}^{\times}$-actions satisfying the conditions of Proposition \ref{prop:structureofcompletion}. A morphism of $(\widehat{g}_{\hbar},L)$-modules is a $L$ and $\mathbb{C}^{\times}$ equivariant $\widehat{R_{\hbar}U(\mathfrak{g})}$-module homomorphism. Write $M(\widehat{g}_{\hbar},L)$ for the abelian category of $(\widehat{g}_{\hbar},L)$-modules (with morphisms defined as above) and $\mathrm{HC}(\widehat{g}_{\hbar},L)$ for the full subcategory of $(\widehat{g}_{\hbar},L)$-modules finitely-generated over $\widehat{R_{\hbar}U(\mathfrak{g})}$. Completion defines a functor $M(\mathfrak{g}_{\hbar},K) \to M(\widehat{g}_{\hbar},L)$. Its restriction to the subcategory $\mathrm{HC}(\mathfrak{g}_{\hbar},K)$ is exact.

\begin{prop}[Prop 2.4.1, \cite{Losev2011}]\label{prop:completionexact}
If $X_{\hbar} \in \mathrm{HC}(\mathfrak{g}_{\hbar},K)$, the natural map
$$
\widehat{R_{\hbar}U(\mathfrak{g})} \otimes_{R_{\hbar}U(\mathfrak{g})} X_{\hbar} \to \widehat{X}_{\hbar}$$
is an isomorphism. In particular, $\widehat{X}_{\hbar}$ is a finitely-generated $\widehat{R_{\hbar}U(\mathfrak{g})}$-module. In other words, the completion functor restricts
$$\widehat{\bullet}: \mathrm{HC}(\mathfrak{g}_{\hbar},K) \to \mathrm{HC}(\widehat{\mathfrak{g}}_{\hbar},L).$$
This functor is exact. 
\end{prop}

\begin{cor}\label{cor:completionfacts}
Suppose $X_{\hbar} \in \mathrm{HC}(\mathfrak{g}_{\hbar},K)$. Then
\begin{enumerate}
\item there is a natural isomorphism
$$\widehat{X}_{\hbar}/\hbar \widehat{X}_{\hbar} \simeq \widehat{X_{\hbar}/\hbar X_{\hbar}}.$$
\item $\widehat{X}_{\hbar}=0$ if and only if $e \notin \mathrm{Supp}(X_{\hbar})$.
\end{enumerate}
\end{cor}

\begin{proof}
\begin{enumerate}
\item 
Let $X_{\hbar} \in \mathrm{HC}(\mathfrak{g}_{\hbar},K)$. Multiplication by $\hbar$ defines a short exact sequence in $\mathrm{HC}(\mathfrak{g}_{\hbar},K)$
$$X_{\hbar} \overset{\hbar}{\to} X_{\hbar} \to X_{\hbar}/\hbar X_{\hbar} \to 0.$$
Since completion is exact, we get a short exact sequence in $\mathrm{HC}(\widehat{\mathfrak{g}}_{\hbar},L)$
$$\widehat{X}_{\hbar} \overset{\hbar}{\to} \widehat{X}_{\hbar} \to \widehat{X_{\hbar}/\hbar X_{\hbar}} \to 0,$$
which gives rise to the desired isomorphism. 

\item Use the description of (commutative) completion provided in part (1) of Proposition \ref{prop:4completions} to observe that that $\widehat{X_{\hbar}/\hbar X_{\hbar}} = 0$ if and only if $e \notin \mathrm{Supp}(X_{\hbar}/\hbar X_{\hbar}) =: \mathrm{Supp}(X_{\hbar})$. From the previous part, $\widehat{X_{\hbar}/\hbar X_{\hbar}} = 0$ if and only if $\widehat{X}_{\hbar} = \hbar \widehat{X}_{\hbar}$. From Proposition \ref{prop:completionexact}, $\widehat{X}_{\hbar}$ is a finitely-generated $\widehat{R_{\hbar}U(\mathfrak{g})}$-module and therefore, from Proposition \ref{prop:completionfacts}, separated in the $\widehat{I}_{e}$-adic topology. In particular (since $\hbar \in \widehat{I}_{e}$)
$$\bigcap_n \hbar^n \widehat{X}_{\hbar} =0.$$
Consequently, $\widehat{X}_{\hbar} = \hbar \widehat{X}_{\hbar}$ if and only if $\widehat{X}_{\hbar} = 0$. Putting all of these implications together, we deduce the result.
\end{enumerate}
\end{proof}

For every $\widehat{X}_{\hbar} \in M(\widehat{\mathfrak{g}}_{\hbar},L)$, we will define a special subspace $\Gamma \widehat{X}_{\hbar}$ of $\widehat{X}_{\hbar}$ (actually $\Gamma X_{\hbar}$ is not, strictly speaking, a subspace when $K$ is disconnected. We will address this difficulty in a moment). 

As usual, denote the identity component of $K$ by $K^0$. Let $K^1 = LK^0$. Note that $K^1$ is a subgroup of $K$ with Lie algebra $\mathfrak{k}$ and the component group of $L/(L \cap K^0)$. The construction of $\Gamma \widehat{X}_{\hbar}$ proceeds in stages:

\begin{enumerate}
\item  First, take the subspace $\Gamma^0\widehat{X}_{\hbar}$ of $K^0$-finite vectors. More precisely, define
\begin{align*}
\Gamma^0\widehat{X}_{\hbar} := \{&x \in \widehat{X}_{\hbar}: x \ \text{belongs to a finite-dimensional } \mathfrak{k}\hbar-\text{invariant}\\
&\text{subspace which integrates to a representation of } K^0\}
\end{align*}
Since $K^0$ is connected, $\Gamma^0\widehat{X}_{\hbar}$ has a well-defined algebraic $K^0$-action. It is also an $R_{\hbar}U(\mathfrak{g})$ submodule of $\widehat{X}_{\hbar}$ and the $K^0$-action is compatible with the module structure in the two usual ways. Since $\mathfrak{k}\hbar$ is stable under the $L$ and Kazhdan $\mathbb{C}^{\times}$-actions on $R_{\hbar}U(\mathfrak{g})$, $\Gamma^0\widehat{X}_{\hbar}$ is stable under the $L$ and Kazhdan $\mathbb{C}^{\times}$-actions on $\widehat{X}_{\hbar}$. The $L$-action on $\Gamma^0\widehat{X}_{\hbar}$ is locally-finite---its differential coincides with the locally finite action of $\mathfrak{k}\hbar$. This presents an interesting complication. The module $\Gamma^0\widehat{X}_{\hbar}$ has two (in general, distinct) algebraic actions of $L \cap K^0$. One comes from the $K^0$-action built into the definition of $\Gamma^0 \widehat{X}_{\hbar}$. The other comes from the $L$-action on $\widehat{X}_{\hbar}$. These two actions of $L \cap K^0$ differentiate to the same action of $\mathfrak{l} = \mathrm{Lie}(L \cap K^0)$ and therefore agree on the identity component of $L \cap K^0$.

\item Next, form the subspace of $\Gamma^0\widehat{X}_{\hbar}$ consisting of vectors on which the two $L \cap K^0$-actions coincide. 
$$\Gamma^1\widehat{X}_{\hbar} := \{x \in \Gamma^0\widehat{X}_{\hbar}: l \cdot_1 x = l \cdot_2 x, \quad \forall l \in L \cap K^0\}.$$
This subspace is an $R_{\hbar}U(\mathfrak{g})$-submodule of $\Gamma^0\widehat{X}_{\hbar}$ and is stable under $\mathbb{C}^{\times}$. It has algebraic actions of $L$ and $K^0$ which agree on the intersection and therefore an algebraic action of $K^1=LK^0$. 

\item  Take the $\mathbb{C}^{\times}$-finite vectors in $\Gamma^1\widehat{X}_{\hbar}$.
\begin{align*}
\Gamma^1_{\mathrm{lf}}\widehat{X}_{\hbar} := \{&x \in \Gamma^1\widehat{X}_{\hbar}: x \ \text{belongs to a finite-dimensional }\\
&\mathbb{C}^{\times}-\text{invariant subspace}\}.
\end{align*}
This subspace has the structure of a $R_{\hbar}U(\mathfrak{g})$-module with algebraic actions of $K^1$ and $\mathbb{C}^{\times}$. The $K^1$-action is compatible with the module structure in the two usual ways. The $\mathbb{C}^{\times}$-action is compatible with the module structure in the sense that the action map $R_{\hbar}U(\mathfrak{g}) \otimes \Gamma^1_{\mathrm{lf}}\widehat{X}_{\hbar} \to \Gamma^1_{\mathrm{lf}}\widehat{X}_{\hbar}$ is $\mathbb{C}^{\times}$-equivariant. Note that the actions of $K^1$ and $\mathbb{C}^{\times}$ do not, in general, commute. The actions of $L$ and $\mathbb{C}^{\times}$ obviously do, but the actions of $K^0$ and $\mathbb{C}^{\times}$ do not. We can fix this by composing the existing $\mathbb{C}^{\times}$-action with $\gamma(t)^{-1}$ (in effect, undoing the `Kazhdanification' required to make the original $\mathbb{C}^{\times}$-action lift to the completion). The result is a grading on $\Gamma^1_{\mathrm{lf}}\widehat{X}_{\hbar}$ which is manifestly even. Halve it, to obtain a grading which is compatible (under the natural map $X_{\hbar} \to \widehat{X}_{\hbar}$) with the original grading on $X_{\hbar}$. With this new grading, $\Gamma^1_{\mathrm{lf}}$ has the structure of a graded, $K^1$-equivariant $R_{\hbar}U(\mathfrak{g})$-module (with the standard grading on $R_{\hbar}U(\mathfrak{g})$). 

\item The final step is to perform a finite induction
$$\Gamma \widehat{X}_{\hbar} := \mathrm{Ind}^K_{K^1} \Gamma^1_{\mathrm{lf}}\widehat{X}_{\hbar}.$$
If we identify $\Gamma \widehat{X}_{\hbar}$ with functions 
$$\{f: K \to \Gamma^1_{\mathrm{lf}} \widehat{X}_{\hbar}: f(k'k) = k' \cdot f(k) \ \text{for } k' \in K^1,k \in K\}$$
there is an $R_{\hbar}U(\mathfrak{g})$-module structure on $\Gamma \widehat{X}_{\hbar}$ defined by
$$(Yf)(k) = (k\cdot Y)f(k), \qquad Y \in R_{\hbar}U(\mathfrak{g}), k \in K, f \in \Gamma\widehat{X}_{\hbar}$$
and an algebraic $\mathbb{C}^{\times}$-action defined by
$$(t\cdot f)(k) = t \cdot f(k), \qquad t \in \mathbb{C}^{\times}, k \in K, f \in \Gamma\widehat{X}_{\hbar}$$
Summarizing, $\Gamma \widehat{X}_{\hbar}$ has the structure of a $R_{\hbar}U(\mathfrak{g})$-module with algebraic $K$ and $\mathbb{C}^{\times}$-actions. It is easy to check that these three structures satisfy the defining properties of a $(\mathfrak{g}_{\hbar},K)$-module.
\end{enumerate}
 
Since all of the ingredients used to define $\Gamma \widehat{X}_{\hbar}$ ($K^0$-finite vectors, $\Gamma^1$, $\mathbb{C}^{\times}$-finite vectors, induction) are functorial, $\Gamma$ gives rise to a functor $M(\widehat{\mathfrak{g}}_{\hbar},L) \to M(\mathfrak{g}_{\hbar},K)$. Define
$$\Phi_{e}:= \Gamma \circ \widehat{\bullet}: M(\mathfrak{g}_{\hbar},K) \to M(\widehat{\mathfrak{g}}_{\hbar},K).$$
Clearly, $\Phi_{e}$ is left exact: it is the composition of a completion functor (exact, by Proposition \ref{prop:completionexact}), $K^0$-finite vectors (left exact), $\Gamma^1$ (left exact), $\mathbb{C}^{\times}$-finite vectors (left exact), and finite induction (exact). In Section \ref{sec:cohomology}, we will study its right derived functors. We will need the following elementary result.

\begin{prop}\label{prop:enoughinjectives}
The category $M(\mathfrak{g}_{\hbar},K)$ has enough injectives.
\end{prop}

\begin{proof}
This follows from an easy general fact: suppose $\mathcal{A}$ and $\mathcal{B}$ are abelian categories and $(L:\mathcal{A}\to \mathcal{B}, R: \mathcal{B} \to \mathcal{A})$ is an adjunction. Suppose that $L$ is exact and that the natural map $A \to RLA$ is an injection for every $A \in \mathcal{A}$. Then if $\mathcal{B}$ has enough injectives, so does $\mathcal{A}$.

If $\mathcal{A} = M(\mathfrak{g}_{\hbar},K)$ and $\mathcal{B} = R_{\hbar}U(\mathfrak{g})-\mathrm{mod}$, the forgetful functor $L: \mathcal{A} \to \mathcal{B}$ has a right adjoint $R: \mathcal{B} \to \mathcal{A}$ defined in much the same way as $\Gamma$. If $B \in \mathcal{B}$, the subspace
\begin{align*}
R^0(B) = \{&b \in B: b \ \text{belongs to a finite-dimensional } \mathfrak{k}\hbar - \text{invariant subspace} \\
&\text{which integrates to a representation of } \ K^0\}
\end{align*}
has the structure of a $K^0$-equivariant $R_{\hbar}U(\mathfrak{g})$-module and
$$R^1(B) = \mathrm{Ind}^K_{K^0}R^0B$$
has the structure of a $K$-equivariant $R_{\hbar}U(\mathfrak{g})$-module. We need to force a grading on $R^1B$ (compatible with the $K$-action and the module structure in all of the usual ways). Define
$$R(B) = \bigoplus_{n \in \mathbb{Z}} R^1(B),$$
putting one copy of $R^1B$ in every integer degree. Give $R(B)$ the structure of a $(\mathfrak{g}_{\hbar},K)$-module by defining
$$Y(b_n) = (Yb)_{m+n}, \quad Y \in R_{\hbar}U(\mathfrak{g})_m \qquad 
k(b_n) = (kb)_n, \quad  k \in K \qquad
t(b_n) = (t^nb)_n, \quad t \in \mathbb{C}^{\times}.$$
It is easy to check that $\mathcal{A},\mathcal{B},L$, and $R$ satisfy the conditions listed above. It is well known that $R-\mathrm{mod}$ has enough injectives for any ring $R$. Hence, $\mathcal{A}$ has enough injectives by the general fact above.
\end{proof}

For the remainder of this section, we will enforce the assumption
\begin{equation}\label{eq:codim2}\codim(\partial \OO,\overline{\OO}) \geq 2.\end{equation}
Let $j: \OO \subset \overline{\OO}$ be the inclusion.

\begin{prop}\label{prop:coherencepreserved}
Recall the containments
$$\Coh^{K \times \CC^{\times}}(\overline{\OO}) \subset \Coh^{K \times \CC^{\times}}_{\overline{\OO}}(\mathfrak{g}/\mathfrak{k})^* \subset \mathrm{HC}_{\overline{\OO}}(\mathfrak{g}_{\hbar},K) \subset M(\mathfrak{g}_{\hbar},K)$$
from Proposition \ref{prop:functors}. The functor $\Phi_{e}$ preserves all three subcategories of $M(\mathfrak{g}_{\hbar},K)$. Its restriction to $\Coh^{K \times \CC^{\times}}(\overline{\OO})$ coincides with the functor
$$j_*j^*: \Coh^{K \times \CC^{\times}}(\overline{\OO}) \to \Coh^{K \times \CC^{\times}}(\overline{\OO})$$
\end{prop}

\begin{proof}
Suppose $M \in \Coh^{K, \mathbb{C}^{\times}}(\overline{\OO})$. By the definition of $\Phi_{e}$ and Proposition \ref{prop:geomsignificance}, it is clear that $\Phi_{e}M$ coincides with the $\mathbb{C}^{\times}$-finite part of $j_*j^*M$. But the $\mathbb{C}^{\times}$-action on $j_*j^*M$ is already finite, so $\Phi_{e}M = j_*j^*M$. By Proposition \ref{thm:finitegeneration} and the codimension condition on $\OO$, this is an object in $\Coh^{K \times \CC^{\times}}(\overline{\OO})$. 

Now suppose $M \in \Coh^{K, \mathbb{C}^{\times}}_{\overline{\OO}}(\mathfrak{g}/\mathfrak{k})^*$. $M$ admits a finite filtration by $K$ and $\mathbb{C}^{\times}$-equivariant subsheaves
$$0 = M_0 \subset M_1 \subset ... \subset M_t = M, \qquad N_i :=M_i/M_{i-1} \in \Coh^{K, \mathbb{C}^{\times}}(\overline{\OO}) \ \text{for} \ 1 \leq i \leq t$$
We have $\Phi_{e}M_1 \in \Coh^{K \times \CC^{\times}}_{\overline{\OO}}(\mathfrak{g}/\mathfrak{k})^*$ by the previous paragraph. Suppose $\Phi_{e}M_i \in \Coh^{K \times \CC^{\times}}_{\overline{\OO}}(\mathfrak{g}/\mathfrak{k})^*$ for some $i <t$. There is a short exact sequence
$$0 \to M_{i} \to M_{i+1} \to N_{i+1} \to 0$$
in $\Coh^{K \times \CC^{\times}}_{\overline{\OO}}(\mathfrak{g}/\mathfrak{k})^*$. By the left exactness of $\Phi_{e}$, there is a long exact sequence in $\QCoh^{K \times \CC^{\times}}_{\overline{\OO}}(\mathfrak{g}/\mathfrak{k})^*$
$$0 \to \Phi_{e}M_i \to \Phi_{e}M_{i+1} \to \Phi_{e}N_{i+1} \to ... $$
Note that $\Phi_{e}M_i$ is coherent by hypothesis and $\Phi_{e}N_{i+1}$ is coherent since $N_{i+1} \in \Coh^{K \times \CC^{\times}}(\overline{\OO})$. Hence, $\Phi_{e}M_{i+1}$ is coherent, since it fits within an exact sequence between coherent sheaves. By induction on $i$, $\Phi_{e}M \in \Coh^{K \times \CC^{\times}}_{\overline{\OO}}(\mathfrak{g}/\mathfrak{k})^*$.

Finally, suppose $X_{\hbar} \in \mathrm{HC}(\mathfrak{g}_{\hbar},K)$. Define $M = X_{\hbar}/\hbar X_{\hbar} \in \Coh^{K, \mathbb{C}^{\times}}_{\overline{\OO}}(\mathfrak{g}/\mathfrak{k})^*$. There is a short exact sequence
$$0 \to \hbar X_{\hbar} \to X_{\hbar} \to M \to 0$$
in $\mathrm{HC}_{\overline{\OO}}(\mathfrak{g}_{\hbar},K)$. By the left exactness of $\Phi_{e}$, there is a long exact sequence
$$0 \to \Phi_{e}\hbar X_{\hbar} \to \Phi_{e}X_{\hbar} \to \Phi_{e} \to ... $$
in $M_{\overline{\OO}}(\mathfrak{g}_{\hbar},K)$ and hence an inclusion
$$\Phi_{e}X_{\hbar}/\Phi_{e}\hbar X_{\hbar} \subseteq \Phi_{e}M.$$
It is clear from the construction of $\Phi_{e}$ that $\hbar\Phi_{e}X_{\hbar} = \Phi_{e}\hbar X_{\hbar}$. So we obtain from above
$$\Phi_{e}X_{\hbar}/\hbar \Phi_{e}X_{\hbar} \subseteq \Phi_{e}M.$$
The left hand side is coherent since the right hand side is coherent. Choose a finite set of generators $x_1,...,x_n$ for $\Phi_{e}X_{\hbar}/\hbar \Phi_{e}X_{\hbar}$ over $S(\mathfrak{g}/\mathfrak{k})$. Choose arbitrary lifts $\widetilde{x}_1,...,\widetilde{x}_n$ to $\Phi_{e}X_{\hbar}$ and form the $(\mathfrak{g}_{\hbar},K)$-submodule $R \subset \Phi_{e}$ generated by these elements. By definition, $\Phi_{e}X_{\hbar} = R + \hbar \Phi_{e}X_{\hbar}$. If we replace $X_{\hbar}$ with $\hbar X_{\hbar}$ and repeat the same argument, we obtain $\hbar \Phi_{e}X_{\hbar} = \hbar R + \hbar^2 \Phi_{e}X_{\hbar}$, and hence $\Phi_{e}X_{\hbar} = R + \hbar^2\Phi_{e}X_{\hbar}$. Then, $\Phi_{e}X_{\hbar} = R + \hbar^n\Phi_{e}X_{\hbar}$ by a simple induction on $n$.

Since $R$ is finitely-generated over a nonnegatively graded ring, its grading is bounded from below. Choose an integer $N$ such that $R_n = 0$ for every $n < N$. If $n<N$ and $x \in (\Phi_{e}X_{\hbar})_n$, then $x \in \bigcap_n \hbar^n\Phi_{e}X_{\hbar}$. Since $\widehat{X}_{\hbar}$ is separated in the $\widehat{I}_{e}$-adic topology (part (3) of Proposition \ref{prop:completionfacts}),
$$\bigcap_n \hbar^n \widehat{X}_{\hbar} \subseteq \bigcap_n \widehat{I}_{e}^n\widehat{X}_{\hbar} = 0$$
Then it is clear from the construction of $\Gamma$ that 
$$\bigcap_n \hbar^n \Phi_{e} X_{\hbar} = 0.$$
So, $x=0$ and we see that the grading on $\Phi_{e}X_{\hbar}$ is (also) bounded from below. Now suppose $n$ is arbitrary and $y \in (\Phi_{e}X_{\hbar})_n$. Choose $m$ so large that $(\hbar^m\Phi_{e}X_{\hbar})_n=0$. Then $\Phi_{e}X_{\hbar} = R + \hbar^m \Phi_{e}X_{\hbar}$ implies $y \in R$. This proves that $\Phi_{e}X_{\hbar} = R$. 

Now, $\Phi_{e}X_{\hbar}$ is a finitely-generated $(\mathfrak{g}_{\hbar},K)$-module and hence an object in $\mathrm{HC}(\mathfrak{g}_{\hbar},K)$. From the inclusion $\Phi_{e}X_{\hbar}/\hbar \Phi_{e}X_{\hbar}\subseteq \Phi_{e}M$ and the additivity of support, we have $\Phi_{e}X_{\hbar} \in \mathrm{HC}_{\overline{\OO}}(\mathfrak{g}_{\hbar},K)$. 
\end{proof}

Write $\mathcal{A}_{\overline{\OO}}$ for the full image of the completion functor $\widehat{\bullet}: \mathrm{HC}_{\overline{\OO}}(\mathfrak{g}_{\hbar},K) \to \mathrm{HC}(\widehat{\mathfrak{g}}_{\hbar},L)$. 

\begin{prop}\label{prop:adjunction}
The functors
$$\widehat{\bullet}: \mathrm{HC}_{\overline{\OO}}(\mathfrak{g}_{\hbar},K) \to \mathcal{A}_{\overline{\OO}}, \qquad 
\Gamma:\mathcal{A}_{\overline{\OO}} \to \mathrm{HC}_{\overline{\OO}}(\mathfrak{g}_{\hbar},K)$$
are left and right adjoints.
\end{prop}

\begin{proof}
Both functors factor through the intermediate category  $\mathrm{HC}_{\overline{\OO}}(\mathfrak{g}_{\hbar},K^1)$
\begin{center}
\begin{tikzcd}
\mathrm{HC}(\mathfrak{g}_{\hbar},K) \arrow[r, "\mathrm{Res}",shift left=.5ex] & \mathrm{HC}(\mathfrak{g}_{\hbar},K^1) \arrow[r,"\widehat{\bullet}"
,shift left=.5ex] \arrow[l,"\mathrm{Ind}",shift left=.5ex] & \mathcal{A}_{\overline{\OO}} \arrow[l,"\Gamma^1_{\mathrm{lf}}",shift left=.5ex]
\end{tikzcd}
\end{center}
The functor $\mathrm{Ind}^K_{K^1}: \mathrm{HC}_{\overline{\OO}}(\mathfrak{g}_{\hbar},K^1) \to \mathrm{HC}_{\overline{\OO}}(\mathfrak{g}_{\hbar},K)$, as we have defined it, is left-adjoint to $\mathrm{Res}$. There is an alternative definition of $\mathrm{Ind}^K_{K^1}$ via tensor products and co-invariants (rather than functions and invariants), given by
$$\mathrm{Ind}^K_{K^1} V = \mathbb{C}[K] \otimes_{\mathbb{C}[K^1]} V$$
and this second version of induction is right-adjoint to $\mathrm{Res}$. Since $K^1$ has finite index in $K$, these two versions coincide. Thus, it suffices to exhibit an adjunction between the functors $\widehat{\bullet}: \HC(\fg_{\hbar},K^1) \to \mathcal{A}_{\overline{\OO}}$ and $\Gamma^1_{\mathrm{lf}}:\mathcal{A}_{\overline{\OO}} \to \HC(\fg_{\hbar},K^1)$. 

Choose $X \in \mathrm{HC}_{\overline{\OO}}(\mathfrak{g}_{\hbar},K^1)$ and $Y \in \mathcal{A}_{\overline{\OO}}$. We want to define a natural bijection
$$\mathrm{Hom}_{\mathfrak{g}_{\hbar},K^1,\mathbb{C}^{\times}}(X,\Gamma^1_{\mathrm{lf}}Y) \simeq \mathrm{Hom}_{\widehat{\mathfrak{g}}_{\hbar},L,\mathbb{C}^{\times}}(\widehat{X},Y)$$
Suppose $f \in \mathrm{Hom}_{\mathfrak{g}_{\hbar},K^1,\mathbb{C}^{\times}}(X,\Gamma^1_{\mathrm{lf}}Y)$. Compose $f$ with the inclusion $i:\Gamma^1_{\mathrm{lf}}Y \subset Y$ to obtain an $L$ and $\mathbb{C}^{\times}$-equivariant $R_{\hbar}U(\mathfrak{g})$-module homomorphism $i\circ f: X \to Y$. Since $Y$ is complete in the $\widehat{I}_{e}$-adic topology (part (3) of Proposition \ref{prop:completionfacts}), this homomorphism extends to a (unique) morphism in $\mathrm{HC}(\widehat{\mathfrak{g}}_{\hbar},L)$ 
$$\widehat{i\circ f}: \widehat{X} \to Y.$$
On the other hand, if $g \in \mathrm{Hom}_{\widehat{\mathfrak{g}}_{\hbar},L,\mathbb{C}^{\times}}(\widehat{X},Y)$, the restriction $g|_X$ takes values in $\Gamma^1_{\mathrm{lf}}Y$. One easily checks that the assignments $f \mapsto \widehat{i \circ f}$ and $g \mapsto g|_X$ define mutually inverse bijections. 
\end{proof}

The following proposition establishes some of the corresponding properties of $\Phi_{e}$. The statements and proofs are analagous to \cite[Prop 3.4.1]{Losev2011}.

\begin{prop}\label{prop:allpropsoflocfunctor}
The functor
$$\Phi_{e}: \mathrm{HC}_{\overline{\OO}}(\mathfrak{g}_{\hbar},K) \to \mathrm{HC}_{\overline{\OO}}(\mathfrak{g}_{\hbar},K)$$
which is well-defined by Proposition \ref{prop:coherencepreserved}, has the following properties:
\begin{enumerate}
\item For every $X_{\hbar} \in \mathrm{HC}_{\overline{\OO}}(\mathfrak{g}_{\hbar},K)$, there is a natural map
$$X_{\hbar} \to \Phi_{e}X_{\hbar},$$
and its completion
$$\widehat{X}_{\hbar} \to \widehat{\Phi_{e}X_{\hbar}}$$
is an injection.

\item $\ker{\Phi_{e}} = \mathrm{HC}_{\partial \OO}(\mathfrak{g}_{\hbar},K)$.

\item For every $X_{\hbar} \in \mathrm{HC}_{\overline{\OO}}(\mathfrak{g}_{\hbar},K)$,
$$\mathrm{Ann}(X_{\hbar}) \subseteq \mathrm{Ann}(\Phi_{e}X_{\hbar}).$$

\item Form the right derived functors $R^i\Phi_{e}: M(\mathfrak{g}_{\hbar},K) \to M(\mathfrak{g}_{\hbar},K)$ using Proposition \ref{prop:enoughinjectives}. Then if $X_{\hbar} \in \mathrm{HC}_{\overline{\OO}}(\mathfrak{g}_{\hbar},K)$, the gradings on $R^i\Phi_{e}X_{\hbar}$ are bounded from below. 
\end{enumerate}
\end{prop}

\begin{proof}
\begin{enumerate}
\item This is a formal consequence of the adjunction $(\widehat{\bullet},\Gamma)$ established in Proposition \ref{prop:adjunction}. The natural map $X_{\hbar} \to \Phi_{e}X_{\hbar}$ is the morphism in $\mathrm{Hom}_{\mathfrak{g}_{\hbar},K,\mathbb{C}^{\times}}(X_{\hbar},\Phi_{e}X_{\hbar})$ corresponding to the identity map $\mathrm{id} \in \mathrm{Hom}_{\widehat{\mathfrak{g}}_{\hbar},L,\mathbb{C}^{\times}}(\widehat{X}_{\hbar},\widehat{X}_{\hbar})$. Its completion is a morphism
$$\widehat{X}_{\hbar} \to \widehat{\Phi_{e}X_{\hbar}}$$
in $\mathrm{HC}(\widehat{\mathfrak{g}}_{\hbar},L)$. On the other hand, the identity map $\mathrm{id} \in \mathrm{Hom}_{\mathfrak{g}_{\hbar},K,\mathbb{C}^{\times}}(\Phi_{e}X_{\hbar},\Phi_{e}X_{\hbar})$ corresponds to a natural map $\widehat{\Phi_{e}X_{\hbar}} \to \widehat{X}_{\hbar}$ in $\mathrm{HC}(\widehat{\mathfrak{g}}_{\hbar},L)$. Since all maps are natural, the composition 
$$\widehat{X}_{\hbar} \to \widehat{\Phi_{e}X_{\hbar}} \to \widehat{X}_{\hbar}$$
is the identity. In particular, $\widehat{X}_{\hbar} \to \widehat{\Phi_{e}X_{\hbar}}$ is an injection. 

\item If $X_{\hbar} \in \mathrm{HC}_{\partial \OO}(\mathfrak{g}_{\hbar},K)$, then $\Phi_{e}X_{\hbar} = 0$ by (one half of) part (2) of Proposition \ref{cor:completionfacts}. Conversely, if $X_{\hbar} \in \mathrm{HC}_{\overline{\OO}}(\mathfrak{g}_{\hbar},K)$ and $\Phi_{e}X_{\hbar} =0$, then $\widehat{\Phi_{e}X_{\hbar}} = 0$ and hence $\widehat{X}_{\hbar} = 0$ by the result of the previous part. Then $X_{\hbar} \in \mathrm{HC}_{\partial \OO}(\mathfrak{g}_{\hbar},K)$ by (the other half of) part (2) of Proposition \ref{cor:completionfacts}.

\item Since $\widehat{X}_{\hbar}$ is an inverse limit of quotients $X_{\hbar}/I_{e}^nX_{\hbar}$ --- each annihilated by $\mathrm{Ann}(X_{\hbar})$---there is an obvious inclusion $\mathrm{Ann}(X_{\hbar}) \subseteq \mathrm{Ann}(\widehat{X}_{\hbar})$. And since $\Gamma^1_{\mathrm{lf}}\widehat{X}_{\hbar} \subseteq \widehat{X}_{\hbar}$, we have $\mathrm{Ann}(\widehat{X}_{\hbar}) \subseteq \mathrm{Ann}(\Gamma^1_{\mathrm{lf}}\widehat{X}_{\hbar})$. Examining the formula for the $R_{\hbar}U(\mathfrak{g})$-action on $\mathrm{Ind}^K_{K^1}\Gamma^1_{\mathrm{lf}}\widehat{X_{\hbar}}$, it is clear that $\mathrm{Ann}(\Gamma \widehat{X}_{\hbar})$ is the largest $K$-invariant subspace of $\mathrm{Ann}(\Gamma^1_{\mathrm{lf}}\widehat{X}_{\hbar})$. But $\mathrm{Ann}(X_{\hbar})$ is already $K$-invariant, so $\mathrm{Ann}(X_{\hbar}) \subseteq \mathrm{Ann}(\Phi_{e} X_{\hbar})$.
\item Consider the abelian category $M^{\geq 0}(\mathfrak{g}_{\hbar},K)$ of $(\mathfrak{g}_{\hbar},K)$-modules with nonnegative gradings. In the proof of Proposition \ref{prop:enoughinjectives}, we defined a functor $R: R_{\hbar}U(\mathfrak{g})-\mathrm{mod} \to M(\mathfrak{g}_{\hbar},K)$ right adjoint to the forgetful functor. $R$ was defined by
$$R(B) = \bigoplus_{n \in \mathbb{Z}}R^1(B), \quad B \in R_{\hbar}U(\mathfrak{g})-\mathrm{mod},$$
for $R^1B$ a certain $K$-equivariant $R_{\hbar}U(\mathfrak{g})$-module produced canonically from $B$. We could have defined
$$R(B) = \bigoplus_{n \geq 0}R^1(B), \quad B \in R_{\hbar}U(\mathfrak{g})-\mathrm{mod}$$
This is still a $(\mathfrak{g}_{\hbar},K)$-module since $R_{\hbar}U(\mathfrak{g})$ is nonnegatively graded, and the resulting functor $R: R_{\hbar}U(\mathfrak{g})-\mathrm{mod} \to M^{\geq 0}(\mathfrak{g}_{\hbar},K)$ is right adjoint to the corresponding forgetul functor. Then the general fact cited in the proof of Proposition \ref{prop:adjunction} implies enough injectives in $M^{\geq 0}(\mathfrak{g}_{\hbar},K)$. 

Now, consider the category $M^b(\mathfrak{g}_{\hbar},K)$ of $(\mathfrak{g}_{\hbar},K)$-modules with gradings bounded from below. If $X_{\hbar} \in M^b(\mathfrak{g}_{\hbar},K)$, we can shift the grading on $X_{\hbar}$ by an appropriate integer $N$ to obtain an object $X_{\hbar}^N \in M^{\geq 0}(\mathfrak{g}_{\hbar},K)$. $X_{\hbar}^N$ has an injective covering $X_{\hbar}^N \hookrightarrow I$ in $M^{\geq 0}(\mathfrak{g},K)$ by the result of the previous paragraph. Since the shift $I^{-N} \in M^b(\mathfrak{g}_{\hbar},K)$ remains injective, $X_{\hbar} \hookrightarrow I^{-N}$ is an injective covering of $X_{\hbar}$. Hence, $M^b(\mathfrak{g}_{\hbar},K)$ has enough injectives as well.

Let $X_{\hbar} \in \mathrm{HC}_{\overline{\OO}}(\mathfrak{g}_{\hbar},K)$. $X_{\hbar}$ is finitely-generated over a nonnegatively-graded ring and, therefore, an object of $M^b(\mathfrak{g}_{\hbar},K)$. Choose an injective resolution $0 \to X_{\hbar} \to I^{\bullet}$ in $M^b(\mathfrak{g}_{\hbar},K)$. The result follows from the standard construction of $R^i\Phi_{e}X_{\hbar}$. 

\end{enumerate}
\end{proof}

From these properties, we deduce

\begin{prop}\label{prop:phiislocalization}
$\Phi_{e}$ is a localization functor for the subcategory $\mathrm{HC}_{\partial \OO}(\mathfrak{g}_{\hbar},K) \subset \mathrm{HC}_{\overline{\OO}}(\mathfrak{g}_{\hbar},K)$.
\end{prop}

\begin{proof}
We will apply the general criterion of Proposition \ref{prop:criterionquotient}. We proved in Proposition \ref{prop:adjunction} that the functors
\begin{center}
\begin{tikzcd}
\mathrm{HC}_{\overline{\OO}}(\mathfrak{g}_{\hbar},K) \arrow[r, shift left=.5ex,"\widehat{\bullet}"] & \mathcal{A}_{\overline{\OO}} \arrow[l, shift left=.5ex,"\Gamma"]
\end{tikzcd}
\end{center}
form an adjoint pair and in part $(2)$ of Proposition \ref{prop:allpropsoflocfunctor} that $\ker{\Phi_{e}} = \mathrm{HC}_{\partial \OO}(\mathfrak{g}_{\hbar},K)$. It remains to show that $\Gamma: \mathcal{A}_{\overline{\OO}} \to \mathrm{HC}_{\overline{\OO}}(\mathfrak{g}_{\hbar},K)$ is fully faithful. 

Choose objects $\widehat{X}_{\hbar},\widehat{Y}_{\hbar} \in \mathcal{A}_{\overline{\OO}}$. Suppose $f \in \mathrm{Hom}_{\mathfrak{g}_{\hbar},K,\mathbb{C}^{\times}}(\Gamma \widehat{X}_{\hbar},\Gamma \widehat{Y}_{\hbar})$. Compose $f$ with the natural map $\Gamma \widehat{Y}_{\hbar} \to Y_{\hbar} \to \widehat{Y}_{\hbar}$ to obtain an $L$ and $\mathbb{C}^{\times}$-equivariant $R_{\hbar}U(\mathfrak{g})$-module homomorphism $\Gamma \widehat{X}_{\hbar} \to \widehat{Y}_{\hbar}$. Since $\widehat{Y}_{\hbar}$ is complete in the $\widehat{I}_{e}$-adic topology (Proposition \ref{prop:completionfacts}), this homomorphism extends to a unique morphism $\widetilde{f} \in \mathrm{Hom}_{\widehat{\mathfrak{g}}_{\hbar}, L,\mathbb{C}^{\times}}(\widehat{\Phi_{e}X_{\hbar}},\widehat{Y}_{\hbar})$ making the following diagram commute
\begin{center}
\begin{tikzcd}
\Gamma \widehat{X}_{\hbar} \arrow[r,"f"] \arrow[d] & \Gamma \widehat{Y}_{\hbar} \arrow[d] \\
\widehat{\Phi_{e}X_{\hbar}} \arrow[r,dashrightarrow,"\exists ! \widetilde{f}"] & \widehat{Y}_{\hbar}
\end{tikzcd}
\end{center}
The restriction $\widetilde{f}|_{\widehat{X}_{\hbar}}$ is a morphism in $\mathrm{Hom}_{\widehat{g}_{\hbar},L,\mathbb{C}^{\times}}(\widehat{X}_{\hbar},\widehat{Y}_{\hbar})$ and 
the correspondence $f \mapsto \widetilde{f}|_{\widehat{X}_{\hbar}}$ defines a map $\mathrm{Hom}(\Gamma \widehat{X}_{\hbar}, \Gamma \widehat{Y}_{\hbar}) \to \mathrm{Hom}(\widehat{X}_{\hbar},\widehat{Y}_{\hbar})$ which is manifestly inverse to $\Gamma$. 
\end{proof}

From Proposition \ref{prop:phiislocalization} and the general properties of localization (Proposition \ref{prop:generalpropsoflocalization}), we get a number of additional properties more or less for free:

\begin{cor}
$\Phi_{e}$ has the following additional properties

\begin{enumerate}
\item For every $X_{\hbar} \in \mathrm{HC}_{\overline{\OO}}(\mathfrak{g}_{\hbar},K)$, the kernel and cokernel of the natural map
$$X_{\hbar} \to \Phi_{e} X_{\hbar}$$
are objects in $\mathrm{HC}_{\partial \OO}(\mathfrak{g}_{\hbar},K)$
\item $X_{\hbar} \in \mathrm{Im}\Phi_{e}$ if and only if
$$\mathrm{Hom}(\mathrm{HC}_{\partial \OO}(\mathfrak{g}_{\hbar},K),X_{\hbar})=\mathrm{Ext}^1(X_{\hbar},\mathrm{HC}_{\partial \OO}(\mathfrak{g}_{\hbar},K))=0$$

\item $\widehat{\bullet} \circ \Gamma: \mathcal{A}_{\overline{\OO}} \to \mathcal{A}_{\overline{\OO}}$ is the identity functor. In particular, the injection $\widehat{X}_{\hbar} \hookrightarrow \widehat{\Phi_{e}X_{\hbar}}$ from Proposition \ref{prop:allpropsoflocfunctor}(1) is actually an isomorphism.
\end{enumerate}
\end{cor}

If we choose a different representative $e' \in \OO$, we get a different functor $\Phi_{e'}:\mathrm{HC}_{\overline{\OO}}(\mathfrak{g}_{\hbar},K) \to \mathrm{HC}_{\overline{\OO}}(\mathfrak{g}_{\hbar},K)$. This functor enjoys all of the properties enumerated above. In particular, it is a localization functor for the subcategory $\mathrm{HC}_{\partial \OO}(\mathfrak{g}_{\hbar},K)$. But localization functors are defined by a universal property and are therefore unique up to natural isomorphism. This proves 

\begin{prop}\label{prop:welldefined}
If $e,e' \in \OO$, there is a natural isomorphism
$$\Phi_{e} \simeq \Phi_{e'}$$
We can therefore write $\Phi_{\OO}$ without ambiguity. 
\end{prop}

Recall the embedding
$$R_{\hbar}:\mathrm{HC}^{\mathrm{filt}}_{\overline{\OO}}(\mathfrak{g},K) \hookrightarrow \mathrm{HC}_{\overline{\OO}}(\mathfrak{g}_{\hbar},K)$$
from Proposition \ref{prop:functors}. Setting $\hbar=1$ defines a right inverse to $R_{\hbar}$
$$\hbar=1: \mathrm{HC}_{\overline{\OO}}(\mathfrak{g}_{\hbar},K) \to \mathrm{HC}^{\mathrm{filt}}_{\overline{\OO}}(\mathfrak{g},K),$$
which restricts to an equivalence on the subcategory $\mathrm{HC}^{\mathrm{tf}}_{\overline{\OO}}(\mathfrak{g}_{\hbar},K)$ of $\hbar$-torsion free $(\mathfrak{g}_{\hbar},K)$-modules. The condition of being $\hbar$-torsion free means that
$$0 \to X_{\hbar} \overset{\hbar^n}{\to} X_{\hbar} \quad \text{is exact} \ \forall n \in \mathbb{N}.$$
Since $\Phi_{\OO}$ is left-exact, this condition is preserved under application of $\Phi_{\OO}$. So $\Phi_{\OO}$ preserves the subcategory $\mathrm{HC}^{\mathrm{tf}}_{\overline{\OO}}(\mathfrak{g},K)$. The next proposition is analagous to the discussion at the beginning of \cite[Sec 3.4]{Losev2011}. 

\begin{prop}\label{prop:descenttoHC}
$\Phi_{\OO}$ descends to a well-defined functor on $\mathrm{HC}_{\overline{\OO}}(\mathfrak{g},K)$. More precisely, there is a unique functor
$$\overline{\Phi}_{\OO}: \mathrm{HC}_{\overline{\OO}}(\mathfrak{g},K) \to \mathrm{HC}_{\overline{\OO}}(\mathfrak{g},K)$$
making the following diagram commute
\begin{center}
\begin{tikzcd}
\mathrm{HC}^{\mathrm{tf}}_{\overline{\OO}}(\mathfrak{g}_{\hbar},K) \arrow[r,"\Phi_{\OO}"] \arrow[d, "\sim","\hbar=1"] & \mathrm{HC}^{\mathrm{tf}}_{\overline{\OO}}(\mathfrak{g}_{\hbar},K) \arrow[d,"\sim","\hbar=1"] \\
\mathrm{HC}^{\mathrm{filt}}_{\overline{\OO}} (\mathfrak{g},K) \arrow[d,"\mathrm{forget}"] \arrow[u,"R_{\hbar}",shift left=1ex] & \mathrm{HC}^{\mathrm{filt}}_{\overline{\OO}} (\mathfrak{g},K) \arrow[d,"\mathrm{forget}"] \arrow[u,"R_{\hbar}",shift left=1ex]\\
\mathrm{HC}_{\overline{\OO}}(\mathfrak{g},K) \arrow[r,"\overline{\Phi}_{\OO}"] &\mathrm{HC}_{\overline{\OO}}(\mathfrak{g},K) 
\end{tikzcd}
\end{center}
\end{prop}

\begin{proof}
First, we will describe how we would \emph{like} to define $\overline{\Phi}_{\OO}$. Then we will prove that this definition makes sense. Define the functor
$$P= \mathrm{forget} \circ (\hbar=1) \circ \Phi_{\OO} \circ R_{\hbar}: \mathrm{HC}^{\mathrm{filt}}_{\overline{\OO}}(\mathfrak{g},K) \to \mathrm{HC}_{\overline{\OO}}(\mathfrak{g},K)$$
For an object $X \in \mathrm{HC}_{\overline{\OO}}(\mathfrak{g},K)$, we would like to define
\begin{equation}\label{eqn:objects}
\overline{\Phi}_{\OO}X = P(X,\mathcal{F}) 
\end{equation}
for any choice of good filtration $\mathcal{F}$. For a morphism $f: X \to Y$ in $\mathrm{HC}_{\overline{\OO}}(\mathfrak{g},K)$ we would like to define
\begin{equation}\label{eqn:morphisms}
\overline{\Phi}_{\OO}f: P(f: (X,\mathcal{F}) \to (Y,\mathcal{G}))
\end{equation}
for any choice of good filtrations $\mathcal{F}$ on $X$ and $\mathcal{G}$ on $Y$ compatible with $f$. There are several things to prove. 

\textbf{Objects.} If $\mathcal{F}$ is a good filtration on $X$ and $s$ is an integer, write $\mathcal{F}^s$ for the filtration defined by $\mathcal{F}^s_iX=\mathcal{F}_{s+i}X$. $\mathcal{F}^s$ is good. If $s \geq t$ there is an identity map 
$$\mathrm{id}_{\mathcal{F}^s,\mathcal{F}^t}:(X,\mathcal{F}^s) \to (X,\mathcal{F}^t),$$
and it is clear from the construction of $\Phi_{\OO}$ that $P(\mathrm{id}_{\mathcal{F}^s,\mathcal{F}^t})$ is the identity. 

Now let $\mathcal{F}$ and $\mathcal{G}$ be arbitrary good filtrations on $X$. There are integers $r \leq s \leq t \leq w$ such that for every integer $i$
\begin{equation}\label{eqn:goodfiltrations}
\mathcal{F}_{i+r}X \subseteq \mathcal{G}_{i+s}X \subseteq \mathcal{F}_{i+t}X \subseteq \mathcal{G}_{i+w}.
\end{equation}
For a proof of this simple fact, see \cite[Prop 2.2]{Vogan1991}. So the identity map defines morphisms
\begin{align*}
&\mathrm{id}_{\mathcal{F}^r, \mathcal{G}^s}: (X,\mathcal{F}^r) \to (X,\mathcal{G}^s)\\
&\mathrm{id}_{\mathcal{G}^s, \mathcal{F}^t}: (X,\mathcal{G}^s) \to (X,\mathcal{F}^t)\\
&\mathrm{id}_{\mathcal{F}^t, \mathcal{F}^w}: (X,\mathcal{F}^t) \to (X,\mathcal{G}^w)
\end{align*}
in $\mathrm{HC}^{\mathrm{filt}}_{\overline{\OO}}(\mathfrak{g},K)$. From the previous paragraph and the functoriality of $P$ we have
\begin{align*}
&P(\mathrm{id}_{\mathcal{G}^s, \mathcal{F}^t}) \circ  P(\mathrm{id}_{\mathcal{F}^r, \mathcal{G}^s}) = P(\mathrm{id}_{\mathcal{G}^s, \mathcal{F}^t} \circ \mathrm{id}_{\mathcal{F}^r, \mathcal{G}^s}) = P(\mathrm{id}_{\mathcal{F}^r,\mathcal{F}^t}) = \mathrm{id}\\
&P(\mathrm{id}_{\mathcal{F}^t, \mathcal{G}^w}) \circ  P(\mathrm{id}_{\mathcal{G}^s, \mathcal{F}^t}) = P(\mathrm{id}_{\mathcal{F}^s, \mathcal{G}^w} \circ \mathrm{id}_{\mathcal{G}^s, \mathcal{F}^t}) = P(\mathrm{id}_{\mathcal{G}^s,\mathcal{G}^w}) = \mathrm{id}
\end{align*}
Hence, $P(\mathrm{id}_{\mathcal{G}^s,\mathcal{F}^t}); P(X,\mathcal{G}^s) \to P(X,\mathcal{F}^t)$ is an isomorphism. But $P(X,\mathcal{F}^t) = P(X,\mathcal{F})$ and $P(X,\mathcal{G}^s) = P(X,\mathcal{G})$. So in fact $P(X,\mathcal{F}) \simeq P(X,\mathcal{G})$. Note that this isomorphism is independent of $r,s,t$, and $w$. Thus, the isomorphisms identifying $P(X,\mathcal{F})$ and $P(X,\mathcal{G})$ are well-defined.

\textbf{Morphisms.} Suppose $f: X \to Y$ is a morphism in $\mathrm{HC}_{\overline{\OO}}(\mathfrak{g},K)$. Choose two different lifts $f_{\mathcal{F},\mathcal{G}}: (X,\mathcal{F}) \to (Y,\mathcal{G})$ and $f_{\mathcal{F}',\mathcal{G}'}: (X,\mathcal{F}') \to (Y,\mathcal{G}')$ to $\mathrm{HC}^{\mathrm{filt}}_{\overline{\OO}}(\mathfrak{g},K)$. We hope to show that 
$$P(f_{\mathcal{F},\mathcal{G}}) = P(f_{\mathcal{F}',\mathcal{G}'})$$
up to the isomorphisms $P(X,\mathcal{F}) \simeq P(X,\mathcal{F}')$ and $P(Y,\mathcal{G}) \simeq P(Y,\mathcal{G}')$ constructed above.

From (\ref{eqn:goodfiltrations}), there are integers $r$ and $s$ such that the identity maps on $X$ and $Y$ induce morphisms
\begin{align*}
&\mathrm{id}_{\mathcal{F}^r,\mathcal{F}}: (X,\mathcal{F}^r) \to (X,\mathcal{F})\\
&\mathrm{id}_{\mathcal{F}^r,\mathcal{F}'}: (X,\mathcal{F}^r) \to (X,\mathcal{F}')\\
&\mathrm{id}_{\mathcal{G},\mathcal{G}^s}: (X,\mathcal{G}) \to (X,\mathcal{G}^s)\\
&\mathrm{id}_{\mathcal{G}',\mathcal{G}^s}: (X,\mathcal{G}') \to (X,\mathcal{G}^s)
\end{align*}
$P(\mathrm{id}_{\mathcal{F}^r,\mathcal{F}})$ and $P(\mathrm{id}_{\mathcal{G},\mathcal{G}^s})$ are the identity maps (on $X$ and $Y$, respectively), and $P(\mathrm{id}_{\mathcal{F}^r,\mathcal{F}'})$ and $P(\mathrm{id}_{\mathcal{G}',\mathcal{G}^s})$ are isomorphisms. The isomorphisms $P(X,\mathcal{F}) \simeq P(X,\mathcal{F}')$ and $P(Y,\mathcal{G}) \simeq P(Y,\mathcal{G}')$ obtained from these maps coincide with the isomorphisms constructed above. By the functoriality of $P$, $P(f_{\mathcal{F},\mathcal{G}}) = P(f_{\mathcal{F}',\mathcal{G'}})$ up to these isomorphisms. 
\end{proof}

As one might expect, $\overline{\Phi}_{\OO}$ inherits all of the essential properties of $\Phi_{\OO}$. Combining Proposition \ref{prop:allpropsoflocfunctor} with Proposition \ref{prop:descenttoHC}, we easily deduce the following.

\begin{prop}\label{prop:allpropertiesofdescent}
The functor
$$\overline{\Phi}_{\OO}: \mathrm{HC}_{\overline{\OO}}(\mathfrak{g},K) \to \mathrm{HC}_{\overline{\OO}}(\mathfrak{g},K),$$
which is well-defined by Proposition \ref{prop:descenttoHC}, has the following properties:
\begin{enumerate}
\item $\overline{\Phi}_{\OO}$ is left exact.
\item $\ker{\overline{\Phi}_{\OO}} = \mathrm{HC}_{\partial \OO}(\mathfrak{g},K)$
\item $\overline{\Phi}_{\OO}$ is a localization functor for the subcategory $\mathrm{HC}_{\partial \OO}(\mathfrak{g},K) \subset \mathrm{HC}_{\overline{\OO}}(\mathfrak{g},K)$
\item There is a natural transformation $\mathrm{id} \to \overline{\Phi}_{\OO}$
\item For every $X \in \mathrm{HC}_{\overline{\OO}}(\mathfrak{g},K)$
$$\mathrm{Hom}(\mathrm{HC}_{\partial \OO}(\mathfrak{g},K),X)=\mathrm{Ext}^1(X,\mathrm{HC}_{\partial \OO}(\mathfrak{g},K))=0.$$
\item For every $X \in \mathrm{HC}_{\overline{\OO}}(\mathfrak{g},K)$,
$$\mathrm{Ann}(X) \subseteq \mathrm{Ann}(\overline{\Phi}_{\OO}X).$$
In particular, $\overline{\Phi}_{\OO}$ preserves central character.
\end{enumerate}
\end{prop}

In short, $\overline{\Phi}_{\OO}$ is a left exact endofunctor of $\mathrm{HC}_{\overline{\OO}}(\mathfrak{g},K)$ which annihilates the subcategory $\HC_{\partial \OO}(\fg,K)$. It is (in a precise sense) a quantum analogue of the localization functor $j_*j^*: \Coh^{K \times \CC^{\times}}(\overline{\OO}) \to  \Coh^{K \times \CC^{\times}}(\overline{\OO})$. 

\subsection{Irreducibility of the associated variety}
As an application of the construction in the previous subsection, we will provide here an alternative proof of Theorem \ref{thm:AV}(iv).

\begin{prop}\label{prop:altirredassvar}
Let $X$ be an irreducible $(\mathfrak{g},K)$-module and let $\OO_G$ denote the (unique) open $G$-orbit in $\AV(\mathrm{Ann}(X))$. If $\AV(X)$ is reducible, then
$$\codim(\partial \OO_G,\overline{\OO}_G) =2.$$
\end{prop}

\begin{proof}
Suppose 
\begin{equation}\label{eq:codim4}\codim(\partial \OO_G,\overline{\OO}_G) \geq 4.\end{equation}
We will show that $\AV(X)$ is irreducible using (a variant of) the functor $\Phi_{\OO}$ constructed in the previous subsection. Let $\OO$ be an open $K$-orbit in $\AV(X)$. Note that (\ref{eq:codim4}) implies
$$\codim(\partial \OO, \overline{\OO}) \geq 2,$$
see (ii) of Theorem \ref{thm:AV}. By Proposition \ref{prop:coherencepreserved}, $\Phi_{\OO}$ restricts to an endofunctor of $\mathrm{HC}_{\overline{\OO}}(\mathfrak{g}_{\hbar},K)$. The key input was Proposition \ref{prop:geomsignificance} combined with Proposition \ref{thm:finitegeneration}. In proposition \ref{prop:geomsignificance}, we took $U$ to be an open dense subset of the ambient variety $X$. If we assume only that $U$ is dense in a component, we can prove a similar result (by exactly the same methods). Namely, we can exhibit a natural isomorphism
$$\Gamma M \simeq j_*j^*M$$
for every $ M \in \Coh^{K \times \CC^{\times}}(\overline{\OO})$. The sheaf $j_*j^*M$ is coherent (by Proposition \ref{thm:finitegeneration}) and supported in $\overline{U}$. Repeating the proof of Proposition \ref{prop:coherencepreserved}, we see that $\Phi_{\OO}$ restricts to a functor
$$\Phi_{\OO}: \mathrm{HC}_{\AV(X)}(\mathfrak{g}_{\hbar},K) \to \mathrm{HC}_{\overline{\OO}}(\mathfrak{g}_{\hbar},K) $$
This descends to a functor
$$\overline{\Phi}_{\OO}: \mathrm{HC}_{\mathrm{AV(X)}}(\mathfrak{g},K) \to \mathrm{HC}_{\overline{\OO}}(\mathfrak{g},K)$$
by a version of Proposition \ref{prop:descenttoHC}. By (4) of Proposition \ref{prop:allpropertiesofdescent}, there is a natural morphism of Harish-Chandra modules
$$X \to \Phi_{\OO}X$$
which is manifestly injective since $X$ is irreducible. We deduce that $\AV(X) \subseteq \AV(\Phi_{\OO}X) \subseteq \overline{\OO}$ and hence that $\AV(X) = \overline{\OO}$. This completes the proof.
\end{proof}

\section{Vogan's conjecture and the cohomology of admissible vector bundles}\label{sec:cohomology}

In this section we will show that, under a slightly stronger codimension condition, Vogan's conjecture (cf. Conjecture \ref{conj:Vogan}) follows from a purely geometric condition on admissible vector bundles. 

\begin{theorem}\label{thm:Vogangeom}
Let $\OO_G \subset \cN$ be a nilpotent orbit such that
\begin{equation}\label{eq:codim6}\codim(\partial \OO_G, \overline{\OO}_G) \geq 6,\end{equation}
and suppose $X$ is a unipotent Harish-Chandra module (cf. Definition \ref{def:specialunipotent}) such that 
$$\AV(\mathrm{Ann}(X)) = \overline{\OO}_G.$$
Then by Theorem \ref{thm:AV}(iv), 
$$\AV(X) = \overline{\OO},$$
for some $K$-orbit in $\OO_G \cap (\fg/\fk)^*$  and by Theorem \ref{thm:unipotentadmissible}
$$\mathrm{AC}(X) = ([\mathcal{E}]),$$
where $\mathcal{E}$ is an admissible $K$-equivariant vector bundle on $\OO$. We can assume without loss that $\mathcal{E}$ is completely reducible. If $H^1(\OO,\mathcal{E})=0$, then as representations of $K$
$$X \simeq_K \Gamma(\OO,\mathcal{E})$$
\end{theorem}

The cohomological condition on the admissibile vector bundle $\mathcal{E}$ in Theorem \ref{thm:Vogangeom} can be verified in many cases. In Section \ref{sec:mainresult}, we will show that if $G_{\RR}$ is complex, then this condition is always satisfied (provided (\ref{eq:codim6}) holds). In \cite{masonbrownthesis}, we show that this condition holds for the $K$-forms of the so-called `model orbit' for $G_{\RR}=\mathrm{Sp}(2n,\RR)$ (i.e. the orbit corresponding to the partition $(2^n)$ of $2n$).

\begin{lemma}\label{lem:R10}
Take $X$ as in Theorem \ref{thm:Vogangeom}, and let 
$$M :=\gr(X) \in \Coh^{K \times \CC^{\times}}(\fg/\fk)^* \subset M(\fg_{\hbar},K).$$
Then the following are true
\begin{itemize}
    \item[(i)] $\Phi_{\OO}M \simeq \Gamma(\OO,\mathcal{E})$ as $K$-representations.
    \item[(ii)] $R^1\Phi_{\OO}M = 0$.
\end{itemize}
\end{lemma}

\begin{proof}
Choose a finite filtration by $K \times \CC^{\times}$-equivariant subsheaves
$$0=M_0 \subset M_1 \subset ... \subset M_n = M$$
such that $N_i:=M_k/M_{k-1} \in \Coh^{K \times \CC^{\times}}(\overline{\OO})$ for $1 \leq k \leq n$. Write $\mathcal{E}_i := N_k|_{\OO}$, a $K \times \CC^{\times}$-equivariant vector bundle on $\OO$. By the definition of the associated $K$-cycle (cf. (\ref{eq:defAC})), we have
$$[\mathcal{E}] = \sum_{k=1}^n [\mathcal{E}_i]$$
in $K\mathrm{Vec}^K(\OO)$. Refining the filtration if necessary, we can further assume that $\mathcal{E}_k$ is irreducible, for each $k$, and hence a direct summand in $\mathcal{E}$ (since $\mathcal{E}$ is assumed to be semisimple). 

We will prove both (i) and (ii) by induction on $n$. If $n=0$, there is nothing to prove. Assume (i) and (ii) hold up to $n-1$. There is a short exact sequence in $\Coh^{K \times \CC^{\times}}(\fg/\fk)^*$
$$0 \to M_{n-1} \to M \to N_n \to 0,$$
and hence a long exact sequence in $M(\fg_{\hbar},K)$
\begin{equation}\label{eq:les1}0 \to \Phi_{\OO}M_{n-1} \to \Phi_{\OO}M \to \Phi_{\OO}N_n \to R^1\Phi_{\OO}M_{n-1} \to R^1\Phi_{\OO}M \to R^1\Phi_{\OO}N_n \to ...\end{equation}
Let $\mathcal{E}' := \bigoplus_{k=1}^{n-1}\mathcal{E}_k$. By the induction hypothesis,
$$\Phi_{\OO}M_{n-1} \simeq_K \Gamma(\OO,\mathcal{E}'), \qquad R^1\Phi_{\OO}M_{n-1}=0.$$
By Proposition \ref{prop:coherencepreserved}, the restriction of $\Phi_{\OO}$ to the subcategory $\Coh^{K \times \CC^{\times}}(\overline{\OO}) \subset M(\fg_{\hbar},K)$ coincides with the functor $j_*j^*$, which in turn coincides (since $\overline{\OO}$ is affine) with the global sections functor $\Gamma(\OO,\bullet)$. Hence $\Phi_{\OO}N_n \simeq_K \Gamma(\OO,\mathcal{E}_n)$ and $R^1\Phi_{\OO}N_n \simeq H^1(\OO,\mathcal{E}_n)$. Note that $H^1(\OO,\mathcal{E}_n)=0$, since $\mathcal{E}_n$ is a direct summand in $\mathcal{E}$ and $H^1(\OO,\mathcal{E})=0$. So $R^1\Phi_{\OO}N_n=0$. Now (\ref{eq:les}) becomes
$$0 \to \Phi_{\OO}M_{n-1} \to \Phi_{\OO}M \to \Phi_{\OO}N_n \to 0\to R^1\Phi_{\OO}M \to 0 \to ... $$
Exactness implies that $R^1\Phi_{\OO}M=0$, proving (ii). For (i), consider the short exact sequence 
$$0 \to \Phi_{\OO}M_{n-1} \to \Phi_{\OO}M \to \Phi_{\OO}N_n \to 0$$
Since $K$ is reductive, this sequence splits upon restriction to $K$. Hence
$$\Phi_{\OO}M \simeq_K \Phi_{\OO}M_{n-1} \oplus \Phi_{\OO}N_n \simeq \Gamma(\OO,\mathcal{E}') \oplus \Gamma(\OO,\mathcal{E}_n) \simeq_K \Gamma(\OO, \mathcal{E}),$$
as desired.
\end{proof}

We are now prepared to prove Theorem \ref{thm:Vogangeom}.

\begin{proof}[Proof of Theorem \ref{thm:Vogangeom}]
The proof has two steps. First, we show that there is a natural isomorphism in $M(\fg,K)$
\begin{equation}\label{eq:step1}X \simeq \overline{\Phi}_{\OO}X.\end{equation}
Then we show that 
\begin{equation}\label{eq:step2}\gr (\overline{\Phi}_{\OO}X) \simeq_K \Gamma(\OO,\mathcal{E})\end{equation}
Since $K$ is reductive, $X \simeq_K \gr(X)$. So (\ref{eq:step1}) and (\ref{eq:step2}) imply
$$X \simeq_K \gr(X) \simeq_K \gr(\overline{\Phi}_{\OO}X) \simeq_K \Gamma(\OO,\mathcal{E}),$$
as asserted.

{\it Step 1} By (4) of Proposition \ref{prop:allpropertiesofdescent}, there is a natural map
\begin{equation}\label{isom:natisom}
\eta: X \to \overline{\Phi}_{\OO}X
\end{equation}
Since $X$ is irreducible, $\eta$ is injective. Let $Y$ be its cokernel. By (6) of Proposition \ref{prop:allpropertiesofdescent} $\mathrm{Ann}(X) \subseteq \mathrm{Ann}(Y)$. If $\mathrm{Ann}(X)=\mathrm{Ann}(Y)$, then by Theorem \ref{thm:AV}, $\dim \AV(X) = \dim \AV(Y)$. Yet by (2) of Proposition \ref{prop:allpropertiesofdescent}, $\mathrm{AV}(Y) \subseteq \partial \OO$. So $\mathrm{Ann}(X) \subsetneq \mathrm{Ann}(Y)$. Since $X$ is unipotent, $\mathrm{Ann}(X)$ is a maximal ideal. So in fact $\mathrm{Ann}(Y) = U(\mathfrak{g})$. Hence, $Y=0$ and (\ref{isom:natisom}) is an isomorphism.

{\it Step 2} Choose a good filtration on $X$ and let $X_{\hbar} = R_{\hbar}X \in \HC^{\mathrm{tf}}_{\overline{\OO}}(\fg_{\hbar},K)$. Write $M = X_{\hbar}/\hbar X_{\hbar} = \gr(X) \in \Coh_{\overline{\OO}}^{K \times \CC^{\times}}(\fg/\fk)^*$. Since $X_{\hbar}$ is $\hbar$-torsion free, multipliciation by $\hbar$ gives rise to a short exact sequence in $M(\mathfrak{g}_{\hbar},K)$
$$0 \to X_{\hbar} \overset{\cdot \hbar}{\to} X_{\hbar} \to M \to 0 $$
Since $\Phi_{\OO}$ is left-exact, there is an associated long exact sequence in $M(\mathfrak{g}_{\hbar},K)$
\begin{equation}\label{eqn:mainseq}
0 \to \Phi_{\OO} X_{\hbar} \overset{\cdot \hbar}{\to} \Phi_{\OO}X_{\hbar} \to \Phi_{\OO}M \to  R^1\Phi_{\OO} X_{\hbar} \overset{\cdot \hbar}{\to} R^1\Phi_{\OO}X_{\hbar} \to R^1\Phi_{\OO}M \to ...\end{equation}
By (ii) of Lemma \ref{lem:R10}, we have $R^1\Phi_{\OO}M = 0$. So $R^1\Phi_{\OO}X_{\hbar} \overset{\cdot \hbar}{\to} R^1\Phi_{\OO}X_{\hbar}$ is surjective. By (4) of Proposition \ref{prop:allpropsoflocfunctor}, the grading on $R^1\Phi_{\OO}X_{\hbar}$ is bounded from below, and multiplication by $\hbar$ increases degree. So $R^1\Phi_{\OO}X_{\hbar} = \hbar R^1\Phi_{\OO}X_{\hbar}$ implies that $R^1\Phi_{\OO}X_{\hbar}=0$. Now (\ref{eqn:mainseq}) becomes
\begin{equation}\label{eqn:mainresult}
0 \to \Phi_{\OO} X_{\hbar} \overset{\cdot \hbar}{\to} \Phi_{\OO}X_{\hbar} \to \Phi_{\OO}M \to 0
\end{equation}
In other words, there is an isomorphism
$$\Phi_{\OO}X_{\hbar}/\hbar \Phi_{\OO}X_{\hbar} \simeq \Phi_{\OO}M$$
in $\Coh^{K \times \CC^{\times}}(\fg/\fk)^*$. The left hand side is identified with the associated graded of the filtered Harish-Chandra module $\Phi_{\OO}X_{\hbar}/(\hbar-1)\Phi_{\OO}X_{\hbar}$. There is an isomorphism of $(\fg,K)$-modules (immediate from the definition of $\overline{\Phi}_{\OO}$)
$$\Phi_{\OO}X_{\hbar}/(\hbar-1)\Phi_{\OO}X_{\hbar} \simeq \overline{\Phi}_{\OO}X.$$
Thus
$$[\gr (\overline{\Phi}_{\OO}X)] = [\Phi_{\OO}M]$$
in $K\Coh^K(\fg/\fk)^*$. By (i) of Lemma \ref{lem:R10},  $\Phi_{\OO}M \simeq_K \Gamma(\OO,\mathcal{E})$. So 
$$\gr (\overline{\Phi}_{\OO}X) \simeq_K \Gamma(\OO,\mathcal{E}).$$
\end{proof}

\section{Vogan's conjecture for complex groups}\label{sec:mainresult}

Let $G_{\RR}$ be a complex connected reductive algebraic group, regarded as a real group by restriction of scalars. In this case, one can make several standard identifications (see \cite[Introduction]{BarbaschVogan1985}): 
\begin{align*}
    G &\simeq G_{\RR} \times G_{\RR}\\
    \cN &\simeq \cN_{\RR} \times \cN_{\RR}\\
    K &\simeq \{(g,g) \in G_{\RR} \times G_{\RR}\}\\
    \cN_{\fk} &\simeq \{(\lambda, \lambda) \in \cN_{\RR} \times \cN_{\RR}\}
\end{align*}

In particular, every $K$-orbit $\OO \subset \cN_{\fk}$ can be $K$-equivariantly identified with a nilpotent co-adjoint $K$-orbit and is therefore a symplectic variety. We will prove the following result.

\begin{prop}\label{prop:vanishing}
Suppose $G_{\RR}$ is complex. Let $\OO \subset \cN_{\fk}$ be a $K$-orbit and let $\mathcal{E}$ be an admissible $K$-equivariant vector bundle on $\OO$. Then
$$H^i(\OO,\mathcal{E})=0, \qquad 0 < i < \codim(\partial \OO, \overline{\OO})-1.$$
\end{prop}

Combined with Theorem \ref{thm:Vogangeom}, this implies

\begin{cor}\label{cor:Vogancomplex}
Suppose $G_{\RR}$ is complex and let $\OO \subset \cN_{\fk}$ be a $K$-orbit such that 
$$\codim(\partial \OO, \overline{\OO}) \geq 4.$$
Suppose $X$ is a unipotent representation of $G_{\RR}$ such that $\AV(X) = \overline{\OO}$. Then the conclusion of Vogan's conjecture (cf. Conjecture \ref{conj:Vogan}) holds for $X$.
\end{cor}

The proof of Proposition \ref{prop:vanishing} will require some preparation.

\begin{lemma}\label{lemma:cm}
Let $V$ be an affine variety and $U \subset V$ an open subset with complement $Z = V \setminus U$. If $M \in \QCoh(X)$ is Cohen-Macaulay, then

$$H^i(U, M|_U) = 0, \qquad 0<i<\codim(Z,V)-1$$
\end{lemma}

\begin{proof}
Let $H^i_Z(V,M)$ denote the cohomology of $X$ with support in $Z$. There is a long exact sequence
\begin{equation}\label{eq:les}
0 \to H_Z^0(V, M) \to H^0(V, M) \to H^0(U, M|_U) \to ...
\end{equation}
See, e.g., \cite[Thm 9.4]{milne}. Since $V$ is affine, $H^i(V, M) = 0$ for $i >0$. Together with (\ref{eq:les}), this implies 
\begin{equation}\label{isoms}
H^i(U, M|_U) \simeq H_Z^{i+1}(V,M), \qquad i \geq 1
\end{equation}
The vanishing behavior of the cohomology groups $H_Z^i(V, M)$ is controlled by the $Z$-depth of $M$, denoted $\mathrm{depth}_Z(M)$. This is defined to be the length of the longest $M$-regular sequence of functions in the ideal defining $Z$. We have in general (without hypotheses on $V$ or on $M$)
\begin{equation}\label{vanishing}
H_Z^i(V,M) = 0, \ i<\mathrm{depth}_Z(M)
\end{equation}
See, e.g., \cite[Thm 5.8]{Huneke2007}. And for $M$ Cohen-Macaulay
\begin{equation} \label{depthisd}
\mathrm{depth}_Z(M) = \codim(Z,V)
\end{equation}
See, e.g., \cite[Chp 18]{Eisenbud1995}. Combining equations \ref{isoms}, \ref{vanishing}, and \ref{depthisd} proves the result. 
\end{proof}

Our application of Lemma \ref{lemma:cm} will be somewhat indirect. Let $\OO \subset \cN_{\fk}$ be a $K$-orbit, and let $p:\widetilde{\OO} \to \OO$ denote the universal $K$-equivariant cover. Consider the affine varieties $V:=\mathrm{Spec}(\CC[\OO])$ and $\widetilde{V} := \mathrm{Spec}(\CC[\widetilde{\OO}])$. There are open embeddings $\OO \subset V$, $\widetilde{\OO} \subset \widetilde{V}$, and the complement of $\OO$ in $V$ (resp. $\widetilde{\OO}$ in $\widetilde{V}$) is of codimension $\geq 2$. Of course, the covering map $p: \widetilde{\OO} \to \OO$ extends to a finite $K$-equivariant surjection $p': \widetilde{V} \to V$. Thus, there is a commutative square:
\begin{center}
\begin{tikzcd}[column sep=large]
\widetilde{\OO} \arrow[r,hookrightarrow] \arrow{d}{p}
&\widetilde{V} \arrow{d}{p'}\\
\OO \arrow[r,hookrightarrow] & V
\end{tikzcd}
\end{center}
\begin{theorem}\label{thm:normalclosuresarenice}
The varieties $\widetilde{V}$ and $V$ are Cohen-Macaulay.
\end{theorem}

\begin{proof}
By \cite[Thm 3.3]{Hinich1991}, $V$ is Gorenstein with rational singularities. Thus by \cite[Thm 6.2]{Broer1998}, the same is true of $\widetilde{V}$. In characteristic 0, rational singularities implies Cohen-Macaulay, see, e.g., \cite[Thm 5.10]{KollarMori1998}. This completes the proof.
\end{proof}

\begin{proof}[Proof of Proposition \ref{prop:vanishing}]
Recall that $\OO$ can be identified with a nilpotent co-adjoint $K$-orbit. In particular, $\OO$ admits a $K$-equivariant symplectic form $\tau$. The pullback $p^*\tau$ of $\tau$ along the covering map $p: \widetilde{\OO} \to \OO$ is a $K$-equivariant symplectic form on $\widetilde{\OO}$. The top exterior power of $p^*\tau$ defines a global trivialization of the canonical bundle $\omega_{\widetilde{\OO}}$ on $\widetilde{\OO}$. So the admissibility condition of Definition \ref{def:admissibility2} becomes
\begin{equation}\label{eq:admissiblecomplex}p^*\mathcal{E} \otimes p^*\mathcal{E} \simeq \mathcal{O}_{\widetilde{\OO}} \oplus ... \oplus \mathcal{O}_{\widetilde{\OO}}.\end{equation}
Choose $e \in \OO$ and write $\rho:K_e \to \mathrm{GL}(E)$ for the $K_e$-representation on the fiber $E$ over $e$, see the remarks preceding Theorem \ref{thm:unipotentadmissible}. Then (\ref{eq:admissiblecomplex}) is equivalent to the condition $2d\rho = 0$. In other words, $\rho$ descends to a representation of the (finite) component group $K_e/K_e^{\circ}$. 

Let $d:=\codim(\OO,V)=\codim(\widetilde{\OO},\widetilde{V})$. By Theorem \ref{thm:normalclosuresarenice}, $\widetilde{V}$ is Cohen-Macaulay. So Lemma \ref{lemma:cm}, applied to $\widetilde{\OO} \subset \widetilde{V}$, implies that 
$$H^i(\widetilde{\OO}, \mathcal{O}_{\widetilde{\OO}})=0, \qquad 0 < i < d-1.$$
Consider the $K$-equivariant vector bundle $p_*\mathcal{O}_{\widetilde{\OO}}$ on $\OO$. Since $p$ is finite, and hence affine,
\begin{equation}\label{eq:H10}H^i(\OO,p_*\mathcal{O}_{\widetilde{\OO}}) \simeq H^i(\widetilde{\OO}, \mathcal{O}_{\widetilde{\OO}})=0, \qquad 0 < i < d-1\end{equation}
Under the equivalence $\mathrm{Vec}^K(\OO) \simeq \mathrm{Rep}(K_e)$, the vector bundle $p_*\mathcal{O}_{\widetilde{V}}$ corresponds to the $K_e$-representation $\CC[K_e/K_e^{\circ}]$. By the representation theory of finite groups, every irreducible representation of $K_e/K_e^{\circ}$ appears as a direct summand in $\CC[K_e/K_e^{\circ}]$. So every irreducible admissible vector bundle on $\OO$ appears as a direct summand in $p_*\mathcal{O}_{\widetilde{\OO}}$. The vector bundle $\mathcal{E}$ is a direct sum of such irreducibles. So 
$$H^i(\OO,\mathcal{E})=0, \qquad 0< i < d-1,$$
as asserted.
\end{proof}

\appendix

\section{Homogeneous Vector Bundles}

Let $K$ be an algebraic group acting on an affine variety $V = \mathrm{Spec}(R)$. Suppose $V$ contains an open, dense $K$-orbit $j:U \subset V$ and let $Z=V \setminus U$. Choose a point $x \in U$ and let $H = K^x$. Since $K$ acts transitively on $U$, a $K$-equivariant coherent sheaf $M \in \Coh^K(U)$ is (the sheaf of sections of) a homogeneous vector bundle, and we will speak interchangeably of $K$-equivariant vector bundles and $K$-equivariant coherent sheaves on $U$. The geometric fiber of $M$ over $x$ is a finite-dimensional vector space carrying a natural action of $H$. On the other hand, if $E$ is a finite-dimensional $H$-representation, there is a $K$-equivariant vector bundle $K \times_H E \to U$ with fiber equal to $E$. It is formed as the quotient space of $K \times E$ under the natural right $H$-action $h \cdot (k,v) = (kh,h^{-1}v)$. Taking the fiber over $x$ and forming the vector bundle $K \times_H E$ define mutually inverse equivalences between $\Coh^K(U)$ and the category of finite-dimensional $H$-representations. 

Define the subgroup $K^1 = K^0H$ and let $i: U^x \subset U$ be the connected component of $x$. We can describe $U^x$ as a homogeneous space in two different ways

\begin{lemma}\label{lemma:concomp}
The following are true:
\begin{enumerate}
\item $K^0$ acts transitively on $U^x$ with isotropy $H \cap K^0$. 
\item $K^1$ acts transitively on $U^x$ with isotropy $H$.
\end{enumerate}
\end{lemma}

\begin{proof}
\begin{enumerate}
\item 
Clearly $K^0x \subseteq U^x$, since $K^0$ is connected. Conversely, suppose $y \in U^x$. Then there is a path connecting $x$ to $y$ in $U$. By the path-lifting property for homogeneous spaces, there is a group element $k \in K$ such that $kx = y$ and a path from $1$ to $k$ in $K$ lifting the path from $x$ to $y$ in $U$. In particular, $k \in K^0$. Therefore, $U^x \subseteq K^0x$. 
\item The orbit $K^1x$ is equal to $K^1/H$ as a homogeneous space for $K^1$. There is an exact sequence of component groups
$$\pi^0(H) \to \pi^0(K^1) \to \pi^0(K^1x) \to 1$$
If we identify $\pi^0(H) = H/H^0$ and $\pi^0(K^1) = K^0H/K^0 = H/(H \cap K^0)$, the leftmost homomorphism is induced by the inclusion $H^0 \subseteq H \cap K^0$. In particular, it is surjective. Therefore, $\pi^0(K^1x)=1$, i.e. $K^1x$ is connected. This provides an inclusion $K^1x \subseteq U^x$. The reverse inclusion follows from $(1)$: $U^x = K^0x \subseteq K^1x$.
\end{enumerate}
\end{proof}

Let $p: \widetilde{U}^x \to U^x$ be the universal $K^0$-equivariant cover. If we choose a lift $\widetilde{x} \in \widetilde{U}^x$ of $x$, then $(K^0)^{\widetilde{x}} = (H \cap K^0)^0 = H^0$. Let $\widetilde{V}$ be the normalization of $V$ in the function field of $\widetilde{U}^x$. Then $\widetilde{V}$ is a normal affine variety with an algebraic action of $K^0$, an open $K^0$-equivariant immersion $\widetilde{U}^x \subset \widetilde{V}$, and a finite, $K^0$-equivariant map $p': \widetilde{V} \to V$ extending the map $p:\widetilde{U}^x \to U^x$:
\begin{center}
\begin{tikzcd}[column sep=large]
\widetilde{U}^x \arrow[r,hookrightarrow] \arrow{d}{p}
&\widetilde{V} \arrow{d}{p'}\\
U^x \arrow[r,hookrightarrow,"j\circ i"] & V\\
\end{tikzcd}
\end{center}
Let $M \in \Coh^K(V)$. Write $\mathcal{E} = M|_U \in \Coh^K(U)$ and $E$ for the fiber over $x$ (a finite-dimensional representation of $H$). Let $\mathfrak{m}_x \subset R$ be the maximal ideal defining $x$ and form the completion of $M$ with respect to $\mathfrak{m}_x$
$$\widehat{M} = \varprojlim M/\mathfrak{m}_x^nM$$
Note that $\widehat{M}$ is a module for the completed algebra $\widehat{R} = \varprojlim R/\mathfrak{m}_x^nR$. The actions of $H$ and $\mathfrak{k}$ on $R$ preserve $\mathfrak{m}_x$ and therefore lift to the completion. These structures exhibit the usual compatibility conditions:
\begin{enumerate}
\item The action map $\widehat{R} \otimes \widehat{M} \to \widehat{M}$ is $\mathfrak{k}$ and $H$-equivariant
\item The action map $\mathfrak{g} \otimes \widehat{M} \to \widehat{M}$ is $H$-equivariant
\item The $\mathfrak{g}$-action coincides on $\mathfrak{h}$ with the differentiated action of $H$. 
\end{enumerate}
Here are some basic facts about $\widehat{M}$.

\begin{prop}\label{prop:4completions}
The following are true:
\begin{enumerate}
\item The stalk $M_x$ is a module for the local ring $R_x$. Form the completions $\widehat{R_x}$ and $\widehat{M_x}$ with respect to $\mathfrak{m}_x \subset R_x$.
\begin{align*}
\widehat{R_x} &= \varprojlim R_x/\mathfrak{m}_x^n\\
\widehat{M_x} &= \varprojlim M_x/\mathfrak{m}_x^nM_x
\end{align*}
Then the natural map $\widehat{R} \to \widehat{R_x}$ is an isomorphism of algebras, and the natural map $\widehat{M} \to \widehat{M_x}$ is an isomorphism of $\mathfrak{k}$ and $H$-equivariant $\widehat{R}$-modules.
\item The sections $\Gamma(U,\mathcal{E})$ form a $K$-equivariant $R$-module. One can define the completion
$$\widehat{\Gamma(U,\mathcal{E})} = \varprojlim \Gamma(U,\mathcal{E})/\mathfrak{m}_x^n\Gamma(U,\mathcal{E})$$
The natural map
$$\widehat{M} \to \widehat{\Gamma(U,\mathcal{E})}$$
is an isomorphism of $\mathfrak{k}$ and $H$-equivariant $\widehat{R}$-modules.
\item The sections $\Gamma(U^x,i^*\mathcal{E})$ form a $K^0$-equivariant $R$-module. One can define the completion
$$\widehat{\Gamma(U^x,i^*\mathcal{E})} = \varprojlim \Gamma(U^x,i^*\mathcal{E})/\mathfrak{m}_x^n\Gamma(U^x,i^*\mathcal{E})$$
The natural map 
$$\widehat{M} \to \widehat{\Gamma(U^x,i^*\mathcal{E})}$$
is an isomorphism of $\mathfrak{k}$ and $H \cap K^0$-equivariant $\widehat{R}$-modules.
\item The sections $\Gamma(\widetilde{U}^x,p^*i^*\mathcal{E})$ form a $K^0$-equivariant $R$-module. One can define the completion
$$\widehat{\Gamma(\widetilde{U}^x,p^*i^*\mathcal{E})} = \varprojlim \Gamma(\widetilde{U}^x,i^*\mathcal{E})/\mathfrak{m}_x^n\Gamma(\widetilde{U}^x,p^*i^*\mathcal{E})$$
The natural map 
$$\widehat{M} \to \widehat{\Gamma(\widetilde{U}^x,p^*i^*\mathcal{E})}$$
is an isomorphism of $\mathfrak{k}$ and $H^0$-equivariant $\widehat{R}$-modules.
\end{enumerate}
\end{prop}

\begin{proof}
\begin{enumerate}
\item By the exactness of localization, there are canonical isomorphisms
$$R_x/\mathfrak{m}_x^n \simeq (R/\mathfrak{m}_x^n)_x, \qquad n \geq 0$$
But every element of $R/\mathfrak{m}_x^n$ outside of $\mathfrak{m}_x/\mathfrak{m}_x^n$ is already a unit, so $(R/\mathfrak{m}_x^n)_x = R/\mathfrak{m}_x^n$. The isomorphisms $R_x/\mathfrak{m}_x^n \simeq R/\mathfrak{m}_x^n$ give rise to an isomorphism $\widehat{R} \simeq \widehat{R_x}$. The isomorphism $\widehat{M} \simeq \widehat{M_x}$ is obtained in a similar manner. 

\item Form the quasi-coherent sheaf $j_*\mathcal{E}$. There is a natural map $M \to j_*\mathcal{E}$, which restricts to an isomorphism over $U$. Hence, the map of sheaves $M \to j_*\mathcal{E}$ induces an isomorphism of stalks
$$M_x \simeq (j_*\mathcal{E})_x = \Gamma(U,\mathcal{E})_x$$
and therefore an isomorphism of $\mathfrak{k}$ and $H$-equivariant $\widehat{R}$-modules
$$\widehat{M_x} \simeq \widehat{\Gamma(U,\mathcal{E})_x}$$
But we saw in $(1)$ that $\widehat{M} \simeq \widehat{M_x}$. The same argument shows that $\widehat{\Gamma(U,\mathcal{E})} \simeq \widehat{\Gamma(U,\mathcal{E})_x}$. Composing all of the isomorphism in sight, we obtain $\widehat{M} \simeq \widehat{\Gamma(U,\mathcal{E})}$ as desired.
\item Repeat the proof for $(2)$, replacing $U$ with $U^x$, $\mathcal{E}$ with $i^*\mathcal{E}$, and $H$ with $H \cap K^0$.
\item Pulling back germs defines a $\mathfrak{k}$ and $H^0$-equivariant $R$-module homomorphism
$$M_x \to (p^*i^*\mathcal{E})_{\widetilde{x}}$$
which is an isomorphism because $p$ is a covering. We complete to obtain an isomorphism $\widehat{M_x} \simeq \widehat{(p^*i^*\mathcal{E})_{\widetilde{x}}}$ of $\mathfrak{k}$ and $H^0$-equivariant $\widehat{R}$-modules. 

We saw in $(1)$ that $\widehat{M} \simeq \widehat{M_x}$. By the remarks after Lemma \ref{lemma:concomp}, $\widetilde{U}^x$ embeds as an open subset in an affine variety $\widetilde{V}$. Denote by $s$ the inclusion of $\widetilde{U}^x$ into $\widetilde{V}$. Then $s_*p^*i^*\mathcal{E}$ is a quasi-coherent sheaf on $\widetilde{V}$ with global sections $\Gamma(\widetilde{U}^x, p^*i^*\mathcal{E})$. In particular, $(p^*i^*\mathcal{E})_{\widetilde{x}} \simeq \Gamma(\widetilde{U}^x, p^*i^*\mathcal{E})_{\widetilde{x}}$ and therefore, $\widehat{(p^*i^*\mathcal{E})_{\widetilde{x}}} \simeq \widehat{\Gamma(\widetilde{U}^x, p^*i^*\mathcal{E})_{\widetilde{x}}}$. Applying part $(1)$ to the sheaf $s_*p^*i^*\mathcal{E}$ provides an isomorphism $\widehat{\Gamma(\widetilde{U}^x, p^*i^*\mathcal{E})} \simeq \widehat{\Gamma(\widetilde{U}^x, p^*i^*\mathcal{E})_{\widetilde{V}}}$. Composing all of the isomorphisms in sight, we obtain $\widehat{M} \simeq \widehat{\Gamma(\widetilde{U}^x,p^*i^*\mathcal{E})}$, as desired.
\end{enumerate}
\end{proof}

We will define a series of modules (all but the last will be subspaces of $\widehat{M}$) in analogy with the modules $\Gamma^0\widehat{X}_{\hbar},\Gamma^1\widehat{X}_{\hbar},\Gamma X_{\hbar}$ defined in Section \ref{sec:quantloc}. 

\begin{enumerate}
\item  First, form the subspace $\Gamma^0\widehat{M}$ of $K^0$-finite vectors:
\begin{align*}
\Gamma^0\widehat{M} = \{&x \in \widehat{M}: x \ \text{belongs to a finite-dimensional } \mathfrak{k}\text{-invariant}\\
&\text{subspace which integrates to a representation of } K^0\}
\end{align*}
Since $K^0$ is connected, $\Gamma^0\widehat{M}$ has a well-defined algebraic $K^0$-action. It is also a $R$-submodule of $\widehat{M}$ and the $K^0$-action is compatible with the module structure in the two usual ways. Since the $\mathfrak{k}$-action on $\widehat{M}$ is $H$-equivariant, $\Gamma^0\widehat{M}$ is invariant under $H$. Hence, $\Gamma^0\widehat{M}$ has two (in general, distinct) actions of $H \cap K^0$, restricted from $H$ and $K^0$, respectively. 

\item Next, form the subspace $\Gamma^1\widehat{M}$ of $\Gamma^0\widehat{M}$ consisting of vectors on which the two $H \cap K^0$-actions coincide
$$\Gamma^1\widehat{M} = \{x \in \Gamma^0\widehat{M}: l \cdot_1 x = l \cdot_2 x, \quad l \in L \cap K^0\} $$
This subspace is an $R$-submodule of $\Gamma^0\widehat{M}$. It has algebraic actions of $H$ and $K^0$ which agree on the intersection and therefore an algebraic action of $K^1=HK^0$. 

\item Finally, induce up to $K$
$$\Gamma \widehat{M} = \mathrm{Ind}^K_{K^1} \Gamma^1\widehat{M}$$
If we identify $\Gamma \widehat{M}$ with functions 
$$\{f: K \to \Gamma^1\widehat{M}: f(k'k) = k' \cdot f(k) \ \text{for } k' \in K^1,k \in K\}$$
there is a natural $R$-module structure on $\Gamma \widehat{M}$ defined by the formula
$$(Yf)(k) = \mathrm{Ad}(b)(Y)f(k), \qquad Y \in R_{\hbar}U(\mathfrak{g}), k \in K, f \in \Gamma\widehat{M}$$
It is easy to check that the action map $R \otimes \Gamma \widehat{M} \to \widehat{M}$ is $K$-equivariant. 
\end{enumerate}

These three modules have geometric significance.

\begin{prop}\label{prop:geomsignificance}
There are natural isomorphisms
\begin{enumerate}
\item $\Gamma^0\widehat{M} \simeq \Gamma(\widetilde{U}^x,p^*i^*\mathcal{E})$ of $K^1$-equivariant $R$-modules,
\item $\Gamma^1\widehat{M} \simeq \Gamma(U^x,i^*\mathcal{E})$ of $K^1$-equivariant $R$-modules,
\item $\Gamma\widehat{M} \simeq \Gamma(U,\mathcal{E})$ of $K$-equivariant $R$-modules.
\end{enumerate}
\end{prop}

We will need a certain `Mackey' isomorphism.

\begin{lemma}\label{lemma:mackey}
Let $W$ and $T$ be finite-dimensional $K$ and $H$-representations, respectively. To simplify the notation, write $I(T)$ for the $K$-equivariant $R$-module $\Gamma(U,K \times_H T)$ and $C(I(T))$ for its completion at $x$. Then $C(I(T))$ is a $\mathfrak{k}$ and $H$-equivariant $\widehat{R}$-module. There is a natural isomorphism of $\mathfrak{k}$ and $H$-equivariant $\widehat{R}$-modules
$$\mathrm{Hom}_{\mathbb{C}}(W,C(I(T))) \simeq C(I(\mathrm{Hom}_{\mathbb{C}}(W,T)))$$
\end{lemma}

\begin{proof}
We know from Mackey that $I$ commutes with $\mathrm{Hom}$ (see, e.g. \cite[Thm 2.95]{KnappVogan1995}). Furthermore, $C$ commutes with $\mathrm{Hom}$ for general abstract reasons: $C$ is a colimit and $\mathrm{Hom}$ is a left-adjoint. Combining these two facts gives the desired isomorphism.
\end{proof}

\begin{proof}[Proof of Proposition \ref{prop:geomsignificance}]
\begin{enumerate}
\item There is a natural injection of $\mathfrak{k}$ and $H^0$-equivariant $\widehat{R}$-modules
$$\alpha: \Gamma(\widetilde{U}^x,p^*i^*\mathcal{E}) \hookrightarrow \widehat{M}$$
obtained by composing the inclusion $\Gamma(\widetilde{U}^x,p^*i^*\mathcal{E}) \subset \widehat{\Gamma(\widetilde{U}^x,p^*i^*\mathcal{E})}$ with the isomorphism of Proposition \ref{prop:4completions}.4. Since $\Gamma(\widetilde{U}^x,p^*i^*\mathcal{E}) = (\mathbb{C}[K^0] \otimes E)^{H^0} \subset \mathbb{C}[K^0] \otimes E$, which is locally-finite as a $K^0$-representation, we have in fact
$$\alpha: \Gamma(\widetilde{U}^x,p^*i^*\mathcal{E}) \hookrightarrow \Gamma^0\widehat{M}$$
Both sides are now locally-finite $K^0$-representations and $\alpha$ is $K^0$-equivariant.

Now, consider the natural projection $r: \widehat{M} \to \widehat{M}/\widehat{\mathfrak{m}_x}\widehat{M} = E$. $r$ is an $H$-equivariant $R$-module homomorphism. Moreover, the restriction of $r$ to $\widehat{M}^{\mathfrak{k}}$ is injective: a kernel element is a taylor series which is both locally constant and $0$-valued at $x$, hence identically $0$. Furthermore, $H$-equivariance implies $r(\widehat{M}^{\mathfrak{k}}) \subseteq E^{\mathfrak{h}}$. Since the composition $p \circ \alpha$ is just evaluation at $\widetilde{x}$
$$(p \circ \alpha)\Gamma(\widetilde{U}^x,p^*i^*\mathcal{E})^{K^0} = E^{H^0} = E^{\mathfrak{h}}$$
So the restriction of $\alpha$ to $K^0$-invariants is an isomorphism
$$\alpha: \Gamma(\widetilde{U}^x, p^*i^*\mathcal{E})^{K^0} \simeq \widehat{M}^{\mathfrak{k}}$$
We will use this fact to show that
$$\mathrm{Hom}_{K^0}(L, \Gamma^0\widehat{M}) = \mathrm{Hom}_{K^0}(L,\Gamma(\widetilde{U}^x,p^*i^*\mathcal{E}))$$
for every finite-dimensional $K^0$-representation $L$. This will imply that the injection $\alpha: \Gamma(\widetilde{U}^x,p^*i^*\mathcal{E}) \hookrightarrow \Gamma^0\widehat{M}$ of algebraic $K^0$-representations is an isomorphism. Note that

\begin{align*}
\mathrm{Hom}_{K^0}(L,\Gamma^0\widehat{M}) &= \mathrm{Hom}_{\mathbb{C}}(L,\widehat{M})^{\mathfrak{k}} &&\text{$K^0$ is connected}\\
						&=\mathrm{Hom}_{\mathbb{C}}(L,  \reallywidehat{\Gamma(\tilde{U}^x, p^*i^*\mathcal{E})})^{\mathfrak{k}} && \text{Proposition \ref{prop:4completions}.4}\\
						&= \reallywidehat{\Gamma(\tilde{U}^x, K^0 \times_{H^0} \mathrm{Hom}_{\mathbb{C}}(L,E))}^{\mathfrak{k}} && \text{Lemma \ref{lemma:mackey}}\\
                        &= \Gamma(\tilde{U}^x,K^0\times_{H^0}\mathrm{Hom}_{\mathbb{C}}(L,E))^{K^0} && \text{Remarks above}\\
                        &= \mathrm{Hom}_{\mathbb{C}}(L,E)^{H^o} && \text{obvious}\\
                        &= \mathrm{Hom}_{H^o}(L,E) && \text{obvious}\\
                        &= \mathrm{Hom}_{K^0}(L, \Gamma(\tilde{U}^x,p^*i^*\mathcal{E})) && \text{Frobenius reciprocity}\\
\end{align*}

\item As explained in the construction of $\Gamma^1\widehat{M}$, $\Gamma^0\widehat{M}$ has two (in general, distinct) algebraic $H \cap K^0$-actions. These actions transfer to $\Gamma(\widetilde{U}^x,p^*i^*\mathcal{E})$ by means of the isomorphism $\Gamma^0\widehat{M} \simeq \Gamma(\widetilde{U}^x,p^*i^*\mathcal{E})$ established above. We need a direct description of these actions on $\Gamma(\widetilde{U}^x,p^*i^*\mathcal{E})$. 

There is, on the one hand, the obvious $K^0$-action on $\widetilde{U}^x$. This induces a $K^0$-action on sections, given by the formula
$$(h \cdot_1 f)(u) = f(h^{-1}u), \qquad f \in \Gamma(\widetilde{U}^x,p^*i^*\mathcal{E}),u \in \widetilde{U}^x, h \in K^0$$
The action of $H$ on $\Gamma(\widetilde{U}^x, p^*i^*\mathcal{E})$ is a bit more subtle. Since $K^0$ is a normal subgroup of $K$, $H$ acts on $\widetilde{U}^x$ by $h(k\widetilde{x}) = (h^{-1}kh)\widetilde{x}$. This induces an $H$-action on sections, given by
$$(h \cdot_2 f)(k\widetilde{x}) = f(h^{-1}kh\widetilde{x}), \qquad f \in \Gamma(\widetilde{U}^x,p^*i^*\mathcal{E}),k \in K^0, h \in H$$
To see that these two actions are the \emph{right} ones (i.e. come from the actions on $\Gamma^0\widehat{M}$ defined in the construction of $\Gamma^1\widehat{M}$) requires a painstaking analysis of the natural isomorphisms in $(1)$. We leave the trivial details to the reader. Now if $f \in \Gamma(\widetilde{U}^x,p^*i^*\mathcal{E})$, then $(h \cdot_1 f) = (h \cdot_2 f)$ for every $h \in H \cap K^0$ if and only if $f\in \Gamma(U^x,i^*\mathcal{E})$. Hence $\Gamma^1\widehat{M} = \Gamma(U^x,i^*\mathcal{E})$.

\item By (2) of Lemma \ref{lemma:concomp}, $\Gamma(U^x,i^*\mathcal{E})$ has the structure of a $K^1$-equivariant $R$-module, and the restriction map
$$\Gamma(U,\mathcal{E}) \to \Gamma(U^x,i^*\mathcal{E})$$
is a homomorphism of $K^1$-equivariant $R$-modules. The $K$-representation $\mathrm{Ind}^K_{K^1}\Gamma(U^x,i^*\mathcal{E})$ has the natural structure of a $K$-equivariant $R$-module as described in the construction of $\Gamma\widehat{M}$, and Frobenius reciprocity provides a natural map of $K$-equivariant $R$-modules
$$r: \Gamma(U, \mathcal{E}) \to \mathrm{Ind}^K_{K^1}\Gamma(U^x,i^*\mathcal{E})$$
If we identify $\mathrm{Ind}^K_{K^1}\Gamma(U^x,i^*\mathcal{E})$ with functions $f:K \to \Gamma(U^x,i^*\mathcal{E})$ satisfying the transformation rule
$$f(k'k)=k'f(k), \qquad k \in K, k' \in K^1$$
then a section $s \in \Gamma(U,\mathcal{E})$ maps to the function $f_s: K \to \Gamma(U^x,i^*\mathcal{E})$ defined by $f_s(k) = (ks)|_{U^x}$. From this description and the equality $KU^x=U$, it is clear that $r$ is an injection. If we can show that $\Gamma(U,\mathcal{E}) \simeq \mathrm{Ind}^K_{K^1}\Gamma(U^x,i^*\mathcal{E})$ as $K$-representations, it will follow that $r$ is an isomorphism. But as $K$-representations
$$\Gamma(U,\mathcal{E}) \simeq_K \mathrm{Ind}^K_{H}E \simeq_K \mathrm{Ind}^K_{K^1}\mathrm{Ind}^{K^1}_H E \simeq_K \mathrm{Ind}^K_{K^1}\Gamma(U^x,i^*\mathcal{E})$$
by the transitivity of induction.
\end{enumerate}
\end{proof}

Finally, there is the question of finite generation.

\begin{theorem}\label{thm:finitegeneration}
Suppose 
$$\codim(Z, V) \geq 2$$
Then the modules $\Gamma(\widetilde{U}^x,p^*i^*\mathcal{E}),\Gamma(U^x,i^*\mathcal{E})$, and $\Gamma(U,\mathcal{E})$ are finitely-generated for $R$.
\end{theorem}

\begin{proof}
Combine Theorems 4.1 and 4.3 in \cite{Grosshans1997}.
\end{proof}

\bibliographystyle{plain}
\bibliography{bibliography.bib}

\end{document}